\documentclass[10.5pt,reqno]{article}
\usepackage{mathrsfs}
\usepackage{amsmath,amssymb,amsthm}
\usepackage[english]{babel}
\newtheorem{thm}{Theorem}[section]
\newtheorem{lem}[thm]{Lemma}
\newtheorem{cor}[thm]{Corollary}

\newtheorem{rem}[thm]{Remark}
\numberwithin{equation}{section}
\title{Global existence and blow-up for the focusing inhomogeneous nonlinear Schr\"{o}dinger equation with inverse-square potential}
\author{{\bf JinMyong An, JinMyong Kim$^*$, RoeSong Jang}\\
\footnotesize{Faculty of Mathematics, {\bf Kim Il Sung} University, Pyongyang, Democratic People's Republic of Korea}\\%address
\footnotesize{$^*$ Corresponding Author: jm.kim0211@ryongnamsan.edu.kp.}
}
\date{}
\begin{document}
\maketitle
\begin{abstract}
In this paper, we study the Cauchy problem for the focusing inhomogeneous nonlinear Schr\"{o}dinger equation with inverse-square potential
\[iu_{t} +\Delta u-c|x|^{-2}u+|x|^{-b} |u|^{\sigma } u=0,\;
u(0)=u_{0} \in H_{c}^{1},\;(t,x)\in \mathbb R\times\mathbb R^{d},\]
where $d\ge3$, $0<b<2$, $\frac{4-2b}{d}<\sigma<\frac{4-2b}{d-2}$ and  $c>-c(d):=-\left(\frac{d-2}{2}\right)^{2}$.
We first establish the criteria for global existence and blow-up of general (not necessarily radial or finite variance) solutions to the equation. Using these criteria, we study the global existence and blow-up of solutions to the equation with general data lying below, at, and above the ground state threshold. Our results extend the global existence and blow-up results of Campos-Guzm\'{a}n (Z. Angew. Math. Phys., 2021) and Dinh-Keraani (SIAM J. Math. Anal., 2021).
\end{abstract}


\section{Introduction}

In this paper, we consider the Cauchy problem for the focusing inhomogeneous nonlinear
Schr\"{o}dinger equation with inverse-square potential, denoted by INLS$_{c}$ equation,
\begin{equation} \label{GrindEQ__1_1_}
\left\{\begin{array}{l} {iu_{t}-P_{c}u+|x|^{-b}
|u|^{\sigma } u=0,\;(t,x)\in\mathbb R\times\mathbb R^{d},}
\\ {u\left(0,\; x\right)=u_{0}(x)\in H_{c}^{1}(\mathbb R^{d}),} \end{array}\right.
\end{equation}
where $d\ge3$, $u:\mathbb R\times \mathbb R^{d} \to \mathbb C$, $u_{0}:\mathbb R^{d} \to \mathbb C$, $0<b<2$, $\sigma >0$ and $P_{c}=-\Delta u+c|x|^{-2}$ with $c>-c(d):=-\left(\frac{d-2}{2}\right)^{2}$.

The INLS$_{c}$ equation appears in a variety of physical settings, for example, in nonlinear optical systems with spatially dependent interactions (see e.g. \cite{BPVT07} and the references therein). In particular, when $b = 0$, the equation \eqref{GrindEQ__1_1_} appears in various areas of physics, for instance in quantum field equations, or in the study of certain black hole solutions of the Einstein equations (see e.g. \cite{BPST03,KSWW75}). When $c = 0$, it appears in nonlinear optics, modeling inhomogeneities in the medium in which the wave propagates (see e.g. \cite{KMVBT17}).

The case $b=c=0$ is the classic nonlinear Schr\"{o}dinger (NLS) equation which has been been widely studied over the last three decades (see e.g. \cite{C03, LP15, WHHG11} and the references therein). The case $b=0$ and $c\neq0$ is known as the NLS equation with inverse-square potential, denoted by NLS$_{c}$ equation, which has also been extensively studied in recent years (see e.g. \cite{D18,KMVZZ17,KMVZ17,MM18, Y21} and the references therein). Moreover, when $c=0$ and $b\neq0$, we have the inhomogeneous nonlinear Schr\"{o}dinger equation, denoted by INLS equation, which has also attracted a lot of interest in recent years (see e.g. \cite{AK21,AKC21,AC21,C21,DK21} and the references therein).

The restriction on $c$ comes from the sharp Hardy inequality:
\begin{equation}\label{GrindEQ__1_2_}
c(d)\int_{\mathbb R^{d}}{|x|^{-2}\left|u(x)\right|^{2}dx}\le\int_{\mathbb R^{d}}{\left|\nabla u(x)\right|dx},~\forall u\in H^{1}(\mathbb R^{d}),
\end{equation}
which ensures that $P_{c}$ is a positive operator. Throughout the paper, for $s\in \mathbb R$ and $1<p<\infty $, we denote by $H^{s,p} (\mathbb R^{d} )$ and $\dot{H}^{s,p} (\mathbb R^{d} )$ the usual nonhomogeneous and homogeneous Sobolev spaces associated to the Laplacian $-\Delta$. As usual, we abbreviate $H^{s,2}(\mathbb R^{d})$ and $\dot{H}^{s,2}(\mathbb R^{d})$ as $H^{s}(\mathbb R^{d})$ and $\dot{H}^{s}(\mathbb R^{d})$, respectively.
Similarly, we define Sobolev spaces in terms of $P_{c}$ via
\begin{equation}\label{GrindEQ__1_3_}
\left\|f\right\|_{\dot{H}_{c}^{s,p}(\mathbb R^{d})}=\left\|(P_{c})^{\frac{s}{2}}f\right\|_{L^{p}(\mathbb R^{d})},~
\left\|f\right\|_{H_{c}^{s,p}(\mathbb R^{d})}=\left\|(1+P_{c})^{\frac{s}{2}}f\right\|_{L^{p}(\mathbb R^{d})}.
\end{equation}
We also abbreviate $\dot{H}_{c}^{s}(\mathbb R^{d})=\dot{H}_{c}^{s,2}(\mathbb R^{d})$ and $H_{c}^{s}(\mathbb R^{d})=H_{c}^{s,2}(\mathbb R^{d})$. If there is no confusion, $\mathbb R^{d}$ will be omitted in various function spaces. Note that by definition, we have
\begin{equation} \label{GrindEQ__1_4_}
\left\|f\right\|_{\dot{H}_{c}^{1}}=\int{|\nabla u(x)|^{2}+c|x|^{-2}|u(x)|^{2}dx.}
\end{equation}
By sharp Hardy inequality (1.2), we see that
\begin{equation} \nonumber
\left\|f\right\|_{\dot{H}_{c}^{1}}\sim \left\|f\right\|_{\dot{H}^{1}}~\textrm{for}~ c>-c(d).
\end{equation}
The INLS$_c$ equation \eqref{GrindEQ__1_1_} is invariant under the scaling,
$$
u_{\lambda}(t,x):=\lambda^{\frac{2-b}{\sigma}}u\left(\lambda^{2}t,\lambda x\right),~\lambda>0.
$$
An easy computation shows that
$$
\left\|u_{\lambda}(0)\right\|_{\dot{H}^{s}}=\lambda^{s-\frac{d}{2}+\frac{2-b}{\sigma}}\left\|u_0\right\|_{\dot{H}^{s}},
$$
which implies that the critical Sobolev index is given by
\begin{equation}\nonumber
s_{c}:=\frac{d}{2}-\frac{2-b}{\sigma}.
\end{equation}
Note that, if $s_{c}=0$ (alternatively $\sigma=\frac{4-2b}{d}$) the problem is known as the mass-critical or $L^{2}$-critical; if $s_{c}=1$ (alternatively $\sigma=\frac{4-2b}{d-2}$) it is called energy-critical or $\dot{H}^1$-critical. The problem is known as intercritical (mass-supercritical and energy-subcritical) if $0<s_{c}<1$ (alternatively $\frac{4-2b}{d}<\sigma<\frac{4-2b}{d-2}$).
On the other hand, solutions to the INLS$_{c}$ equation \eqref{GrindEQ__1_1_} conserve the mass and energy, defined respectively by
\begin{equation} \nonumber
M\left(u(t)\right):=\int_{\mathbb R^{d}}{\left|u(t, x)\right|^{2}dx},
\end{equation}
\begin{equation} \nonumber
E\left(u(t)\right):=\int_{\mathbb R^{d}}{\frac{1}{2}\left|\nabla u(t, x)\right|^{2}+\frac{c}{2}|x|^{-2}\left|u(t, x)\right|^{2}-\frac{1
}{\sigma+2} |x|^{-b} \left|u(t, x)\right|^{\sigma +2} dx}.
\end{equation}
For later uses, it is convenient to introduce the following exponents:
\begin{equation}\label{GrindEQ__1_5_}
\sigma_{c}:=\frac{1-s_{c}}{s_{c}}=\frac{4-2b-(d-2)\sigma}{d\sigma-4+2b},
\end{equation}
\begin{equation}\label{GrindEQ__1_6_}
H_{G}:=E(Q)[M(Q)]^{\sigma_{c}},~ K_{G}:=\left\|Q\right\|_{\dot{H}_{c}^{1}}\left\|Q\right\|_{L^{2}}^{\sigma_{c}},
\end{equation}
where $Q$ is a positive solution to the ground state equation:
\begin{equation}\label{GrindEQ__1_7_}
P_{c}Q+Q-|x|^{-b}|Q|^{\sigma}Q=0.
\end{equation}
Note that the existence of solution to the elliptic equation \eqref{GrindEQ__1_7_} was proved by \cite{CG21}. Moreover, it was proved that all possible solutions to \eqref{GrindEQ__1_7_} have the same mass, the same $\dot{H}^{1}_{c}$ norm and the same energy. See Lemma 2.1. For convenience, we also define the following quantities:
\begin{equation}\label{GrindEQ__1_8_}
M_{G}:=M(Q),~E_{G}:=E(Q),~S_{G}:=\left\|Q\right\|_{\dot{H}_{c}^{1}},~P_{G}:=\int{|x|^{-b}|Q|^{\sigma+2}dx},
\end{equation}
which are the same for all possible solutions to \eqref{GrindEQ__1_7_}.
The quantities $H_{G}$, $K_{G}$, $M_{G}$, $E_{G}$, $S_{G}$ and $P_{G}$ defined in \eqref{GrindEQ__1_6_} and \eqref{GrindEQ__1_8_} appear throughout the paper
and play a important role in our work.

Let us recall the known results for the INLS$_{c}$ equation \eqref{GrindEQ__1_1_}. Using the energy method, Suzuki \cite{S16} showed that if
\footnote[1]{\ Note that the author in \cite{S16} considered \eqref{GrindEQ__1_1_} with $c=-c(d)$. The authors in \cite{CG21} pointed out that the proof for the case $c>-c(d)$ is an immediate consequence of the previous one.}
$d\ge3$, $0<\sigma<\frac{4-2b}{d-2}$, $c>-c(d)$ and $0<b<2$, then the INLS$_{c}$ equation \eqref{GrindEQ__1_1_} is locally well-posed in $H^{1}_{c}$.
Recently, Campos-Guzm\'{a}n \cite{CG21} established the sufficient conditions for global existence and blow-up of $H^{1}_{c}$-solution to \eqref{GrindEQ__1_1_} with $d\ge 3$ and $\frac{4-2b}{d}\le\sigma<\frac{4-2b}{d-2}$.
They also studied the local well-posedness and small data global well-posedness under some assumption on $b$ and $c$ in the energy-subcritical case $\sigma<\frac{4-2b}{d-2}$ with $d\ge 3$, by using the standard Strichartz estimates combined with the fixed point argument.
Furthermore, they showed a scattering criterion and construct a wave operator in $H^{1}_{c}$, for the intercritical case.
Later, the authors in \cite{JAK21} establish the local well-posedness as well as small data global well-posedness and scattering in $H_{c}^{1}$ for the energy-critical INLS$_{c}$ equation \eqref{GrindEQ__1_1_} with $d\ge3$, $c>-\frac{(d+2-2b)^{2}-4}{(d+2-2b)^{2}}c(d)$, $0<b<\frac{4}{d}$ and $\sigma=\frac{4-2b}{d-2}$.
Furthermore, they establish the blow-up criteria for radial or finite variance data.

In this paper, we study the global existence and blow-up of $H_{c}^{1}$-solutions to the focusing INLS$_{c}$ equation \eqref{GrindEQ__1_1_} in the intercritical case $\frac{4-2b}{d}<\sigma<\frac{4-2b}{d-2}$.
As mentioned above, the global existence and blow-up for the focusing intercritical INLS$_{c}$ equation \eqref{GrindEQ__1_1_} were already studied by \cite{CG21}.
More precisely, they proved that the focusing intercritical INLS$_{c}$ equation \eqref{GrindEQ__1_1_} with $d\ge 3$, $0<b<2$  and $c>-c(d)$ is globally well-posed in $H_{c}^1$ if $u_{0}\in H_{c}^1$ satisfies
$$
E(u_{0})[M(u_{0})]^{\sigma_{c}}<H_{G},~
\left\|u_{0}\right\|_{\dot{H}_{c}^{1}}\left\|u_{0}\right\|_{L^{2}}^{\sigma_{c}}
<K_{G}.
$$
They also prove that if
$$
E(u_{0})[M(u_{0})]^{\sigma_{c}}<H_{G},~\left\|u_{0}\right\|_{\dot{H}_{c}^{1}}\left\|u_{0}\right\|_{L^{2}}^{\sigma_{c}}
>K_{G},
$$
and, in addition, if $|x|u_{0}\in L^{2}$ or $u_{0}$ is radial, then the solution blows up in finite time.

As we can see, the global existence and blow up for the focusing intercritical INLS$_{c}$ equation \eqref{GrindEQ__1_1_} were studied only for data below the ground state threshold. In particular, the blow-up result was shown only for radial or finite variance data.
Recently, Dinh-Keraani \cite{DK21} systematically studied the long time dynamics of solutions to the focusing intercritical INLS equation (i.e. \eqref{GrindEQ__1_1_} with $c=0$) with general (not necessarily radial or finite variance) data lying below, at, and above the ground state threshold.

Motivated by \cite{CG21, DK21}, this paper aims to study the global existence and blow-up for the focusing intercritical INLS$_{c}$ equation with general (not necessarily radial or finite variance) data lying below, at, and above the ground state threshold. Our results can be seen as the extension of the global existence and blow-up results of \cite{DK21} for \eqref{GrindEQ__1_1_} with $c=0$  to the focusing intercritical INLS$_{c}$ equation \eqref{GrindEQ__1_1_} with $c>-c(d)$. Our results also extend the ones of \cite{CG21} by treating the general data lying not only below the ground state threshold but also at, and above the ground state threshold.

To this end, we have the following criteria for global existence and blow-up of general solutions to \eqref{GrindEQ__1_1_}.

\begin{thm}\label{thm 1.1.}
Let $d\ge3$, $0<b<2$, $\frac{4-2b}{d}<\sigma<\frac{4-2b}{d-2}$ and $c>-c(d)$. Let $u$ be the solution to \eqref{GrindEQ__1_1_} defined on the maximal forward time interval of existence $[0,T^{*})$.
\begin{enumerate}
 \item \textnormal{(Global existence)} If
  \begin{equation}\label{GrindEQ__1_9_}
   \sup_{t\in [0,T^{*})}{P(u(t))[M(u)]^{\sigma_{c}}}<P_{G}[M_{G}]^{\sigma_{c}},
  \end{equation}
   then $T^{*}=\infty$, where
  \begin{equation}\label{GrindEQ__1_10_}
   P(f):=\int{|x|^{-b}|f(x)|^{\sigma+2}dx}.
  \end{equation}
 \item \textnormal{(Blow-up)} Assume that
  \begin{equation}\label{GrindEQ__1_11_}
   \sup_{t\in [0,T^{*})}{G(u(t))}\le -\delta
  \end{equation}
  for some $\delta>0$, where
  \begin{equation}\label{GrindEQ__1_12_}
   G(f):=\left\|f\right\|_{\dot{H}_{c}^{1}}^{2}-\frac{d\sigma+2b}{2(\sigma+2)}P(f).
  \end{equation}
  Then either $T^{*}<\infty$ or $T^{*}=\infty$ and there exists a time sequence $t_{n}\to \infty$ such that $\left\|u(t_{n})\right\|_{\dot{H}_{c}^{1}}\to \infty$ as $n\to \infty$. Moreover, if we assume in addition that $u$ has finite variance, i.e. $|x|u(t)\in L^{2}$ for all $t\in [0,T^{*})$, then $T^{*}<\infty$.
\end{enumerate}
\end{thm}
A similar statement holds for negative times and we omit the details.

Using Theorem 1.1, we have the following criteria for global existence and blow-up below the ground state threshold.
\begin{thm}[Below the ground state threshold]\label{thm 1.2.}
Let $d\ge3$, $0<b<2$, $\frac{4-2b}{d}<\sigma<\frac{4-2b}{d-2}$ and $c>-c(d)$. Let $u_{0}\in H_{c}^{1}$ satisfy
\begin{equation}\label{GrindEQ__1_13_}
E(u_{0})[M(u_{0})]^{\sigma_{c}}<H_{G}.
\end{equation}
\begin{enumerate}
  \item If $u_{0}$ satisfies
  \begin{equation}\label{GrindEQ__1_14_}
  \left\|u_{0}\right\|_{\dot{H}_{c}^{1}}\left\|u_{0}\right\|_{L^{2}}^{\sigma_{c}}
  <K_{G},
  \end{equation}
  then the corresponding solution to \eqref{GrindEQ__1_1_} satisfies
  \begin{equation}\label{GrindEQ__1_15_}
   \sup_{t\in (-T_{*},\;T^{*})}{P(u(t))[M(u)]^{\sigma_{c}}}<P_{G}[M_{G}]^{\sigma_{c}}.
  \end{equation}
  In particular, the solution exists globally in time.
  \item If $u_{0}$ satisfies
  \begin{equation}\label{GrindEQ__1_16_}
  \left\|u_{0}\right\|_{\dot{H}_{c}^{1}}\left\|u_{0}\right\|_{L^{2}}^{\sigma_{c}}
  >K_{G},
  \end{equation}
  then the corresponding solution to \eqref{GrindEQ__1_1_} satisfies
  \begin{equation}\label{GrindEQ__1_17_}
   \sup_{t\in (-T_{*},\;T^{*})}{G(u(t))}\le -\delta,
  \end{equation}
  for some $\delta>0$. In particular, the solution either blows up in finite time, or there exists a time sequence $(t_{n})_{n\ge 1}$ satisfying $|t_{n}|\to \infty$ such that $\left\|u(t_{n})\right\|_{\dot{H}_{c}^{1}}\to \infty$ as $n\to \infty$. Moreover, if we assume in addition that
  \begin{itemize}
    \item $u_{0}$ has finite variance,
    \item or $u_{0}$ is radially symmetric,
    \item or $\sigma\le 2$, and $u_{0}\in \Sigma_{d}$, where
     \begin{equation}\label{GrindEQ__1_18_}
     \Sigma_{d}:=\left\{f\in H_{c}^{1}:\; f(y,x_{d})=f(|y|,x_{d}),\;x_{d}f\in L^{2}\right\}
     \end{equation}
     with $x=(y,x_{d})$, $y=(x_{1},\cdots,x_{d-1})\in \mathbb R^{d-1}$, and $x_{d}\in \mathbb R$,
  \end{itemize}
  then the corresponding solution blows up in finite time, i.e., $T_{*},\;T^{*}<\infty$.
\end{enumerate}
\end{thm}
\begin{rem}\label{rem 1.3.}
\textnormal{The authors in \cite{CG21} proved the finite time blow-up for radial or finite variance data below the ground state threshold. Theorem 1.2 establishes the blow-up criteria for general (not necessarily radial or finite variance) data. In particular, the finite time blow-up for cylindrically symmetric data is also proved.}
\end{rem}
We also have the following criteria for global existence and blow-up at the ground state threshold.
\begin{thm}[At the ground state threshold]\label{thm 1.4.}
Let $d\ge3$, $0<b<2$, $\frac{4-2b}{d}<\sigma<\frac{4-2b}{d-2}$ and $c>-c(d)$. Let $u_{0}\in H_{c}^{1}$ satisfy
\begin{equation}\label{GrindEQ__1_19_}
E(u_{0})[M(u_{0})]^{\sigma_{c}}=H_{G}.
\end{equation}
\begin{enumerate}
  \item If $u_{0}$ satisfies
  \begin{equation}\label{GrindEQ__1_20_}
  \left\|u_{0}\right\|_{\dot{H}_{c}^{1}}\left\|u_{0}\right\|_{L^{2}}^{\sigma_{c}}
  <K_{G},
  \end{equation}
  then the corresponding solution to \eqref{GrindEQ__1_1_} exists globally in time. Moreover, the solution either satisfies
  \begin{equation}\label{GrindEQ__1_21_}
   \sup_{t\in \mathbb R}{P(u(t))[M(u)]^{\sigma_{c}}}<P_{G}[M_{G}]^{\sigma_{c}}.
  \end{equation}
  or there exists a time sequence $(t_{n})_{n\ge1}$ satisfying $|t_{n}|\to\infty$ such that
  \begin{equation}\label{GrindEQ__1_22_}
  u(t_{n})\to e^{i\theta}Q~\textrm{strongly in}~H_{c}^{1}~\textrm{as}~n\to \infty,
  \end{equation}
  for some $\theta\in\mathbb R$ and some positive solution $Q$ to the ground state equation \eqref{GrindEQ__1_7_}.
  \item If $u_{0}$ satisfies
  \begin{equation}\label{GrindEQ__1_23_}
  \left\|u_{0}\right\|_{\dot{H}_{c}^{1}}\left\|u_{0}\right\|_{L^{2}}^{\sigma_{c}}
  =K_{G},
  \end{equation}
  then $u(t,x)=e^{it}e^{i\theta}Q(x)$ for some $\theta\in \mathbb R$ and some positive solution $Q$ to  \eqref{GrindEQ__1_7_}.
  \item If $u_{0}$ satisfies
  \begin{equation}\label{GrindEQ__1_24_}
  \left\|u_{0}\right\|_{\dot{H}_{c}^{1}}\left\|u_{0}\right\|_{L^{2}}^{\sigma_{c}}
  >K_{G},
  \end{equation}
  then the corresponding solution to \eqref{GrindEQ__1_1_}
  \begin{enumerate}
    \item either blows up forward in time, i.e., $T^{*}<\infty$
    \item or there exists a time sequence $t_{n}\to\infty$ such that $\left\| u(t_{n})\right\|_{\dot{H}_{c}^{1}}\to \infty$ as $n\to\infty$,
    \item or there exists a time sequence $t_{n}\to\infty$ such that \eqref{GrindEQ__1_22_} holds.
  \end{enumerate}
    Moreover, if we assume in addition that
    \begin{itemize}
    \item $u_{0}$ has finite variance,
    \item or $u_{0}$ is radially symmetric,
    \item or $\sigma\le 2$, and $u_{0}\in \Sigma_{d}$,
    \end{itemize}
    then the possibility in Item (b) can be excluded.
\end{enumerate}
\end{thm}
Finally, we study the global existence and blow-up above the ground state threshold.
\begin{thm}[Above the ground state threshold]\label{thm 1.5.}
Let $d\ge 3$, $0<b<2$, $\frac{4-2b}{d}<\sigma<\frac{4-2b}{d-2}$, $c>-c(d)$ and
\begin{equation}\label{GrindEQ__1_25_}
V(t):=\int{|x|^{2}\left|u(t,x)\right|^{2}dx}.
\end{equation}
Let $u_{0}\in H_{c}^{1}$ be such that $|x|u_{0}\in L^2$. Assume that
\begin{equation}\label{GrindEQ__1_26_}
E(u_{0})[M(u_{0})]^{\sigma_{c}}\ge H_{G},
\end{equation}
\begin{equation}\label{GrindEQ__1_27_}
\frac{E(u_{0})[M(u_{0})]^{\sigma_{c}}}{E_{G}[M_{G}]^{\sigma_{c}}}
\left[1-\frac{(V'(0))^{2}}{32E(u_{0})V(0)}\right]\le1.
\end{equation}
\begin{enumerate}
  \item If
  \begin{equation}\label{GrindEQ__1_28_}
  P(u_{0})[M(u_{0})]^{\sigma_{c}}<P_{G}[M_{G}]^{\sigma_{c}},
  \end{equation}
  \begin{equation}\label{GrindEQ__1_29_}
  V'(0)\ge0,
  \end{equation}
  then the corresponding solution to \eqref{GrindEQ__1_1_} satisfies \eqref{GrindEQ__1_9_}. In particular, the solution exists globally in time.
  \item If
  \begin{equation}\label{GrindEQ__1_30_}
  P(u_{0})[M(u_{0})]^{\sigma_{c}}>P_{G}[M_{G}]^{\sigma_{c}},
  \end{equation}
  \begin{equation}\label{GrindEQ__1_31_}
  V'(0)\le0,
  \end{equation}
  then the corresponding solution to \eqref{GrindEQ__1_1_} blows up forward in finite time, i.e., $T^{*}<\infty$.
\end{enumerate}
\end{thm}
\begin{rem}\label{lem 1.6.}
\textnormal{To our knowledge, Theorem 1.4 and Theorem 1.5 are the first global existence and blow-up results for the focusing intercritical INLS$_{c}$ equation \eqref{GrindEQ__1_1_} with data lying at, and above the ground state threshold.}
\end{rem}
\begin{rem}\label{lem 1.7.}
\textnormal{Theorem 1.3, Theorem 1.4 and Theorem 1.5 extend the global existence and blow-up results of \cite{CG21} for focusing INLS equation (i.e. \eqref{GrindEQ__1_1_} with $c=0$) to the focusing INLS$_{c}$ equation \eqref{GrindEQ__1_1_} with $c>-c(d)$.}
\end{rem}
This paper is organized as follows. In Section 2, we prove Theorem 1.1. In Section 3, we study the global existence and blow up of $H_{c}^{1}$-solutions lying below, at, and above the ground state threshold.

%%%%%%%%%%%%%%%%%%%%%%%%%%%%%%%%%%%%%%%%%%%%%%%%%%%%%%%%%%%%%%%%%%%%%%%%%%%%%
\section{Criteria for global existence and blow-up}

In this section, we prove Theorem 1.1. Throughout the paper, $C>0$ will denote positive universal constant, which can be different at different places. $a\lesssim b$ means $a\le Cb$ for some constant $C>0$. We also write $a\sim b$ if $a\lesssim b \lesssim a$.
\subsection{Variational Analysis}
In this subsection, we recall the sharp Gagliardo--Nirenberg inequality related the focusing INLS$_{c}$ equation \eqref{GrindEQ__1_1_}. First of all, we recall the existence of a ground state.
\begin{lem}[Existence of a ground state, \cite{CG21}]\label{lem 2.1.}
Let $d\ge 3$, $0<\sigma<\frac{4-2b}{d-2}$, $c>-c(d)$ and $0<b<2$. There exists a positive solution to the elliptic equation
\begin{equation}\label{GrindEQ__2_1_}
P_{c}Q+Q-|x|^{-b}|Q|^{\sigma}Q=0
\end{equation}
in $H_{c}^{1}$. Moreover, all possible solutions have the same mass $M_{G}$, the same $\dot{H}^{1}_{c}$ norm $S_{G}$ and the same energy $E_{G}$.
\end{lem}
\begin{lem}[Sharp Gagliardo--Nirenberg inequality, \cite{CG21}]\label{lem 2.2.}
Let $d\ge 3$, $0<\sigma<\frac{4-2b}{d-2}$, $c>-c(d)$ and $0<b<2$. Then for $f\in H_{c}^{1}$, we have
\begin{equation}\label{GrindEQ__2_2_}
P(f)\le C_{GN}\left\|f\right\|_{\dot{H}_{c}^{1}}^{\frac{d\sigma+2b}{2}}
\left\|f\right\|_{L^{2}}^{\frac{4-2b-\sigma(d-2)}{2}},
\end{equation}
where $P(f)$ is given in \eqref{GrindEQ__1_10_}. The equality in \eqref{GrindEQ__2_2_} is attained by a function $Q\in H_{c}^{1}$, which is a positive solution to the elliptic equation \eqref{GrindEQ__2_1_}.
\end{lem}
\begin{rem}\label{rem 2.3.}
\textnormal{We also have the following Pohozaev identities:
\begin{equation}\label{GrindEQ__2_3_}
M_{G}=\frac{4-2b-(d-2)\sigma}{d\sigma+2b}
[S_{G}]^{2}=\frac{4-2b-(d-2)\sigma}{2(\sigma+2)}P_{G}.
\end{equation}
In particular, we have
\begin{equation}\label{GrindEQ__2_4_}
C_{GN}=\frac{2(\sigma+2)}{d\sigma+2b}[K_{G}]^{-\frac{d\sigma-4+2b}{2}}.
\end{equation}
We also have
\begin{equation}\label{GrindEQ__2_5_}
E_{G}=\frac{d\sigma-4+2b}{2(d\sigma+2b)}[S_{G}]^{2}
=\frac{d\sigma-4+2b}{4(\sigma+2)}P_{G},
\end{equation}
which implies
\begin{equation}\label{GrindEQ__2_6_}
H_{G}=\frac{d\sigma-4+2b}{2(d\sigma+2b)}
[K_{G}]^{2}.
\end{equation}}
\end{rem}
\subsection{Virial Identities}
Given a real valued function $a$, we define the virial potential by
\begin{equation}\nonumber
V_{a}(t):=\int{a(x)\left|u(t,x)\right|^{2}dx.}
\end{equation}
A simple computation shows that the following result holds.
\begin{lem}[\cite{D18}]\label{lem 2.4.}
Let $d\ge3$ and $c>-c(d)$. If $u:I\times\mathbb R^{d}\to\mathbb C$ is a smooth-in-time and Schwartz-in-space solution to $iu_{t}-P_{c}u=N(u)$, with $N(u)$ satisfying $\textnormal{Im}(N(u)\bar{u})=0$, then we have for any $t\in I$,
\begin{equation}\nonumber
\frac{d}{dt}V_{a}(t)=2\int{\nabla a(x)\cdot \textnormal{Im}(\bar{u}(t,x)\nabla u(t,x))dx},
\end{equation}
and
\begin{eqnarray}\begin{split}\nonumber
\frac{d^{2}}{dt^{2}}V_{a}(t)&=-\int{\Delta^{2}a(x)|u(t,x)|^{2}dx+4\sum_{j,k=1}^{d}
{\int{\partial_{jk}^{2}a(x)\textnormal{Re}(\partial_{k}u(t,x)\partial_{j}\bar{u}(t,x))dx}}}\\
&+4c\int{\nabla a(x)\cdot\frac{x}{|x|^{4}}|u(t,x)|^{2}dx}+2\int{\nabla a(x)\cdot \left\{N(u),u\right\}_{p}(t,x)dx},
\end{split}\end{eqnarray}
where $\left\{f,g\right\}_p:=\textnormal{Re}(f\nabla \bar{g}-g\nabla\bar{f})$ is the momentum bracket.
\end{lem}
Note that if $N(u)=-|x|^{-b}|u|^{\sigma}u$, then
\begin{equation}\nonumber
\left\{N(u),u\right\}_{p}=\frac{\sigma}{\sigma+2}\nabla(|x|^{-b}|u|^{\sigma+2})+\frac{2}{\sigma+2}
\nabla(|x|^{-b})|u|^{\sigma+2}.
\end{equation}
Hence, we immediately have the following result.
\begin{cor}\label{cor 2.5.}
If $u$ is a smooth-in-time and Schwartz-in-space solution to the focusing INLS$_{c}$ equation, then we have for any $t\in I$,

\begin{eqnarray}\begin{split}\nonumber
\frac{d^{2}}{dt^{2}}V_{a}(t)&=-\int{\Delta^{2}a(x)|u(t,x)|^{2}dx+4\sum_{j,k=1}^{d}
{\int{\partial_{jk}^{2}a(x)\textnormal{Re}(\partial_{k}u(t,x)\partial_{j}\bar{u}(t,x))dx}}}\\
&+4c\int{\nabla a(x)\cdot\frac{x}{|x|^{4}}|u(t,x)|^{2}dx}-\frac{2\sigma}{\sigma+2}
\int{\Delta a(x)|x|^{-b}|u(t,x)|^{\sigma+2}dx}\\
&+\frac{4}{\sigma+2}\int{\nabla a(x)\cdot \nabla(|x|^{-b})|u(t,x)|^{\sigma+2}dx}.
\end{split}\end{eqnarray}
\end{cor}

We have the following standard virial identity.
\begin{lem}[Standard Virial Identity, \cite{JAK21}]\label{lem 2.6.}
Let $d\ge 3$, $c>-c(d)$. Let $u_{0}\in H_{c}^{1}$ be such that $|x|u_{0}\in L^2$ and $u:I\times\mathbb R^{d}\to \mathbb C$ be the corresponding solution to the focusing INLS$_{c}$ equation. Then, $|x|u\in C\left(I,\;L^2\right)$. Moreover, for any $t\in I$,
\begin{equation}\nonumber
\frac{d^2}{dt^2}\left\|xu(t)\right\|_{L^{2}}^{2}=8G(u(t)),
\end{equation}
where $G(\cdot)$ is given in \eqref{GrindEQ__1_12_}.
\end{lem}

Let us introduce a function $\theta:[0,\infty)\to [0,\infty)$ satisfying
\begin{equation}\label{GrindEQ__2_7_}
\theta (r)=\left\{\begin{array}{l} {r^2,~\textrm{if}\;0\le r\le 1,}
\\ {\textrm{0},~\textrm{if}\;r\ge 2,} \end{array}\right.
\textrm{and}~~\theta''(r)\le 2 ~~\textrm{for}~~r\ge 0.
\end{equation}
For $R>1$, we define the radial function $\varphi_{R}:\mathbb R^{d}\to [0,\;\infty)$:
\begin{equation}\label{GrindEQ__2_8_}
\varphi_{R}(x)=\varphi_{R}(r):=R^{2}\theta(r/R),\;r=|x|.
\end{equation}
One can easily see that
\begin{equation}\label{GrindEQ__2_9_}
2-\varphi''_{R}(r)\ge0,\;2-\frac{\varphi'_{R}(r)}{r}\ge0,\;2d-\Delta\varphi_{R}(x)\ge0.
\end{equation}

We have the following localized virial estimate.
\begin{lem}[Localized Virial Estimate]\label{lem 2.7.}
Let $d \ge 3$, $0<b<2$, $c>-c(d)$, $R>1$ and $\varphi_{R}$ be as in \eqref{GrindEQ__2_8_}. Let $u:I\times\mathbb R^{d}\to \mathbb C$ be a solution to the focusing INLS$_{c}$ equation \eqref{GrindEQ__1_1_}. Then for any $t\in I$,
\begin{equation} \label{GrindEQ__2_10_}
\frac{d^2}{dt^2}V_{\varphi_{R}}(t)\le 8 G(u(t))+CR^{-2}+CR^{-b}\left\|u(t)\right\|_{H_{c}^{1}}^{\sigma+2},
\end{equation}
where $G$ is given in \eqref{GrindEQ__1_12_}.
\end{lem}
\begin{proof}
Applying Corollary 2.5 with $a(x)=\varphi_{R}(x)$, we have
\begin{eqnarray}\begin{split}\label{GrindEQ__2_11_}
\frac{d^{2}}{dt^{2}}V_{\varphi_{R}}(t)&=-\int{\Delta^{2}\varphi_{R}(x)|u(t,x)|^{2}dx}+4\sum_{j,k=1}^{d}
{\int{\partial_{jk}^{2}\varphi_{R}(x)\textnormal{Re}(\partial_{k}u(t,x)\partial_{j}\bar{u}(t,x))dx}}\\
&+4c\int{\nabla \varphi_{R}(x)\cdot\frac{x}{|x|^{4}}|u(t,x)|^{2}dx}-\frac{2\sigma}{\sigma+2}
\int{\Delta \varphi_{R}(x)|x|^{-b}|u(t,x)|^{\sigma+2}dx}\\
&+\frac{4}{\sigma+2}\int{\nabla \varphi_{R}(x)\cdot \nabla(|x|^{-b})|u(t,x)|^{\sigma+2}dx}.
\end{split}\end{eqnarray}
Noticing that
$$\partial_{j}=\frac{x_{j}}{r}\partial_{r},~\partial^{2}_{jk}=\left(\frac{\delta_{jk}}{r}-\frac{x_{j}x_{k}}{r^{3}}\right)\partial_{r}
+\frac{x_{j}x_{k}}{r^{2}}\partial^{2}_{r},$$
we have
\begin{eqnarray}\begin{split}\label{GrindEQ__2_12_}
&\sum_{j,k=1}^{d}
{\int{\partial_{jk}^{2}\varphi_{R}(x)\textnormal{Re}(\partial_{k}u(t,x)\partial_{j}\bar{u}(t,x))dx}} \\&~~~~~~=\int{\frac{\varphi'_{R}(r)}{r}\left|\nabla u(t,x)\right|^{2}dx}+\int{\left(\frac{\varphi''_{R}(r)}{r^{2}}-\frac{\varphi'_{R}(r)}{r^{3}}\right)
\left|x\cdot\nabla u(t,x)\right|^{2}dx.}
\end{split}\end{eqnarray}
In view of \eqref{GrindEQ__2_11_} and \eqref{GrindEQ__2_12_}, we have
\begin{eqnarray}\begin{split}\label{GrindEQ__2_13_}
\frac{d^{2}}{dt^{2}}V_{\varphi_{R}}(t)&=8G(u(t))-8\left\|u(t)\right\|_{\dot{H}_{c}^{1}}
+\frac{4(d\sigma+2b)}{\sigma+2}\int{|x|^{-b}|u(t,x)|^{\sigma+2}dx}\\
&-\int{\Delta^{2}\varphi_{R}(x)|u(t,x)|^{2}dx}
+4\int{\frac{\varphi'_{R}(r)}{r}\left|\nabla u(t,x)\right|^{2}dx}\\
&+4\int{\left(\frac{\varphi''_{R}(r)}{r^{2}}-\frac{\varphi'_{R}(r)}{r^{3}}\right)\left|x\cdot\nabla u(t,x)\right|^{2}dx}
+4c\int{\frac{\varphi'_{R}(r)}{r^{3}}|u(t,x)|^{2}dx}\\
&-\frac{2\sigma}{\sigma+2}\int{\Delta \varphi_{R}(x)|x|^{-b}|u(t,x)|^{\sigma+2}dx}-\frac{4b}{\sigma+2}\int{|x|^{-b}
\frac{\varphi'(r)}{r}|u(t,x)|^{\sigma+2}dx}.
\end{split}\end{eqnarray}
Since $\left\|\Delta^{2}\varphi_{R}\right\|_{L^{\infty}}\lesssim R^{-2}$, the conservation of mass implies that
\begin{equation}\label{GrindEQ__2_14_}
\left|\int{\Delta^{2}\varphi_{R}(x)|u(t,x)|^{2}dx}\right|\lesssim R^{-2}\left\|u(t)\right\|_{L^2}^{2}\lesssim R^{-2}.
\end{equation}
Using the conservation of mass and the following facts
$$
\left|x\cdot \nabla u\right|\le r|\nabla u|,~\varphi''_{R}(r)\le 2,~\frac{\varphi'_{R}(r)}{r}\le 2,
$$
we see that
\begin{eqnarray}\begin{split}\label{GrindEQ__2_15_}
&4\int{\frac{\varphi'_{R}(r)}{r}\left|\nabla u(t)\right|^{2}dx}
+4\int{\left(\frac{\varphi''_{R}(r)}{r^{2}}-\frac{\varphi'_{R}(r)}{r^{3}}\right)\left|x\cdot\nabla u(t)\right|^{2}dx}+4c\int{\frac{\varphi'_{R}(r)}{r^{3}}|u(t)|^{2}dx}-8\left\|u(t)\right\|_{\dot{H}_{c}^{1}}\\
&~~~~~~~~~~\le 4 \int{\left(2-\frac{\varphi'_{R}(r)}{r}\right)\left(-\left|\nabla u(t)\right|^{2}+\frac{\left|x\cdot\nabla u(t)\right|^{2}}{r^2}-c\frac{|u(t)|}{r^2}\right)dx}\\
&~~~~~~~~~~\le -4c\int{\left(2-\frac{\varphi'_{R}(r)}{r}\right)\frac{|u(t)|}{r^2}dx}
=-4c\int_{|x|>R}{\left(2-\frac{\varphi'_{R}(r)}{r}\right)\frac{|u(t)|}{|x|^2}dx}\\
&~~~~~~~~~~\le \max\left\{-4cS,0\right\} \int_{|x|>R}{R^{-2}|u(t)|^{2}dx}\le \max\left\{-4cSM(u(t)),0\right\}R^{-2}\lesssim R^{-2},
\end{split}\end{eqnarray}
where $S=\max_{r\ge 1}{2-\frac{\theta'(r)}{r}}$.
Furthermore, we have
\begin{eqnarray}\begin{split}\label{GrindEQ__2_16_}
&\frac{4(d\sigma+2b)}{\sigma+2}\int{|x|^{-b}|u(t,x)|^{\sigma+2}dx}-\frac{2\sigma}{\sigma+2}\int{\Delta \varphi_{R}(x)|x|^{-b}|u(t,x)|^{\sigma+2}dx}\\
&~~~~~~~~~~~~~~~~~~~~~~~~~~~~~~~~~~~~~~~~~~
-\frac{4b}{\sigma+2}\int{|x|^{-b}\frac{\varphi'(r)}{r}|u(t,x)|^{\sigma+2}dx}\\
&~~=\frac{2\sigma}{\sigma+2}\int{|x|^{-b}\left(2d-\Delta \varphi_{R}(x)\right)|u(t)|^{\sigma+2}dx}
+\frac{4b}{\sigma+2}\int{|x|^{-b}\left(2-\frac{\varphi'_{R}(r)}{r}\right)|u(t)|^{\sigma+2}dx}\\
&~~\le C\int_{|x|\ge R}{|x|^{-b}|u(t)|^{\sigma+2}}\le CR^{-b}\left\|u(t)\right\|_{L^{\sigma+2}}^{\sigma+2}\le CR^{-b}\left\|u(t)\right\|_{H_{c}^{1}}^{\sigma+2},
\end{split}\end{eqnarray}
where the last inequality follows from the Sobolev embedding and the fact $H_{c}^{1}\sim H^{1}$ as $\sigma<\frac{4-2b}{d-2}$.
In view of \eqref{GrindEQ__2_13_}--\eqref{GrindEQ__2_16_}, we get the desired result.
\end{proof}
\subsection{Proof of Theorem 1.1}
First we prove Item 1. Let $u:[0,T^{*})\times \mathbb R^{d}\to \mathbb C$ be a $H_{c}^{1}$-solution to \eqref{GrindEQ__1_1_} satisfying \eqref{GrindEQ__1_9_}. It follows from \eqref{GrindEQ__1_9_} and the conservation of mass and energy that
\begin{equation}\label{GrindEQ__2_17_}
\sup_{t\in [0,T^{*})}{\left\|u(t)\right\|_{\dot{H}_{c}^{1}}}\le M,
\end{equation}
for some $M<\infty$. This shows that $T^{*}=\infty$.

Next, we prove Item 2. Let $u:[0,T^{*})\times \mathbb R^{d}\to \mathbb C$ be a solution to \eqref{GrindEQ__1_1_} satisfying \eqref{GrindEQ__1_11_}. If $T^{*}<\infty$, then we are done. If $T^{*}=\infty$, then we have to show that there exists $t_{n}\to \infty$ such that $\left\|u(t_{n})\right\|_{\dot{H}_{c}^{1}}\to \infty$ as $n\to \infty$. Assume by contradiction that it doesn't hold, i.e. $\sup_{t\in [0,\infty)}{\left\|u(t)\right\|_{\dot{H}_{c}^{1}}}\le M_{0}$ for some $M_{0}>0$. By the conservation of mass, we have
$$
\sup_{t\in [0,\infty)}{\left\|u(t)\right\|_{H_{c}^{1}}}\le M_{1},
$$
for some $M_{1}>0$. Hence, it follows from Lemma 2.7 (localized virial estimate) and \eqref{GrindEQ__1_11_} that
$$
V''_{\varphi_{R}}(t)\le 8G(u(t))+CR^{-2}+CR^{-b}\left\|u(t)\right\|_{H_{c}^{1}}^{\sigma+2}\le -8\delta+CR^{-2}+CR^{-b}M_{1}^{\sigma+2},
$$
for all $t\in[0,\infty)$. By taking $R>1$ large enough, we have for all $t\in[0,\infty)$,
$$
V''_{\varphi_{R}}(t)\le -4\delta.
$$
Integrating this estimate, there exists $t_{0}>0$ sufficiently large such that $V_{\varphi_{R}}(t_{0})<0$ which is impossible. This completes the proof of the first part of Item 2. Let us assume that $u$ has finite variance. Then it follows from \eqref{GrindEQ__1_11_} and Lemma 2.6 that
$$
\frac{d^2}{dt^2}\left\|xu(t)\right\|_{L^{2}}^{2}=8G(u(t))\le -8\delta.
$$
By the classical argument of Glassey \cite{G77} the solution must blow up in finite time. This completes the proof of Theorem 1.1.
%%%%%%%%%%%%%%%%%%%%%%%%%%%%%%%%%%%%%%%%%%%%%%%%%%%%%%%%%%%%%%%%%%%%%%%%%%%%%

\section{Global existence and blow up below, at, and above the ground state threshold}
In this section, we prove Theorem 1.2, Theorem 1.4 and Theorem 1.5.
%%%%%%%%%%%%%%%%%%%%%%%%%%%%%%%%%%%%%%%%%%%%%%%%%%%%%%%%%%%%%%%%%%%%%%%%%%%%%
\subsection{Below the ground state threshold}
In this subsection, we study the global existence and blow-up below the ground state threshold.
\begin{proof}[\textbf{Proof of Theorem 1.2}]
\eqref{GrindEQ__1_13_} shows that there exists $\delta_{0}>0$ such that
\begin{equation}\label{GrindEQ__3_1_}
E(u_{0})[M(u_{0})]^{\sigma_{c}}\le (1-\delta_{0})H_{G}.
\end{equation}

\textbf{Proof of Item 1.} Let $u_{0}\in H_{c}^{1}$ satisfy \eqref{GrindEQ__1_13_} and \eqref{GrindEQ__1_14_}.
It follows from Lemma 3.5 of \cite{CG21}, \eqref{GrindEQ__3_1_} and the conservation of mass and energy that there exists $\delta>0$ such that
\begin{equation}\label{GrindEQ__3_2_}
\left\|u(t)\right\|_{\dot{H}_{c}^{1}}\left\|u(t)\right\|_{L^{2}}^{\sigma_{c}}\le (1-\delta)K_{G},
\end{equation}
for all $t\in (-T_{*},T^{*})$. Lemma 2.2 and \eqref{GrindEQ__3_2_} yield
\begin{eqnarray}\begin{split}\nonumber
P(u(t))[M(u(t))]^{\sigma_{c}}&\le C_{GN}\left\|u(t)\right\|_{\dot{H}_{c}^1}^{\frac{d\sigma+2b}{2}}
\left\|u(t)\right\|_{L^{2}}^{\frac{4-2b-(d-2)\sigma}{2}+2\sigma_{c}}\\
&=C_{GN}\left(\left\|u(t)\right\|_{\dot{H}_{c}^1}\left\|u(t)\right\|_{L^{2}}^{\sigma_{c}}\right)
^{\frac{d\sigma+2b}{2}}\\
&\le C_{GN}(1-\delta)^{\frac{d\sigma+2b}{2}}
[K_{G}]^{\frac{d\sigma+2b}{2}},
\end{split}\end{eqnarray}
for all $t\in (-T_{*},T^{*})$. By \eqref{GrindEQ__2_4_} and \eqref{GrindEQ__2_3_}, we have
$$
P(u(t))[M(u(t))]^{\sigma_{c}}\le \frac{2(\sigma+2)}{d\sigma+2b}(1-\delta)^{\frac{d\sigma+2b}{2}}
[K_{G}]^{2}=
(1-\delta)^{\frac{d\sigma+2b}{2}}P_{G}[M_{G}]^{\sigma_{c}},
$$
for all $t\in (-T_{*},T^{*})$ which shows \eqref{GrindEQ__1_15_}. By Theorem 1.1, the solution exists globally in time.

\textbf{Proof of Item 2.} Let $u_{0}\in H_{c}^{1}$ satisfy \eqref{GrindEQ__1_13_} and \eqref{GrindEQ__1_16_}. Lemma 3.5 of \cite{CG21} also shows that
\begin{equation}\label{GrindEQ__3_3_}
\left\|u(t)\right\|_{\dot{H}_{c}^{1}}\left\|u(t)\right\|_{L^{2}}^{\sigma_{c}}>
K_{G},
\end{equation}
for all $t\in (-T_{*},T^{*})$. By \eqref{GrindEQ__2_6_}, \eqref{GrindEQ__3_1_}, \eqref{GrindEQ__3_3_} and the conservation of mass and energy, we have
\begin{eqnarray}\begin{split}\nonumber
G(u(t))[M(u(t))]^{\sigma_{c}}&=\left\|u(t)\right\|_{\dot{H}_{c}^{1}}^{2}\left\|u(t)\right\|_{L^{2}}^{2\sigma_{c}}
-\frac{d\sigma+2b}{2(\sigma+2)}P(u(t))[M(u(t))]^{\sigma_{c}}\\
&=\frac{d\sigma+2b}{2}E(u(t))[M(u(t))]^{\sigma_{c}}-\frac{d\sigma-4+2b}{4}
\left(\left\|u(t)\right\|_{\dot{H}_{c}^{1}}\left\|u(t)\right\|_{L^{2}}^{\sigma_{c}}\right)^2\\
&\le \frac{d\sigma+2b}{2}(1-\delta_{0})H_{G}-\frac{d\sigma-4+2b}{4}
[K_{G}]^2\\
&=-\frac{d\sigma-4+2b}{4}\delta_{0}[K_{G}]^2,
\end{split}\end{eqnarray}
for all $t\in (-T_{*},T^{*})$. This shows \eqref{GrindEQ__1_11_} with
$$
\delta:= \frac{(d\sigma-4+2b)\delta_{0}[K_{G}]^{2}}{4[M(u_{0})]^{\sigma_{c}}}>0.
$$
By Theorem 1.1, the corresponding solution either blows up in finite time, or there exists a time sequence $(t_{n})_{n\ge 1}$ satisfying $|t_{n}|\to \infty$ such that $\left\| u(t_{n})\right\|_{\dot{H}_{c}^{1}}\to \infty$ as $n\to \infty$. This completes the proof of the first part of Item 2.

Let us prove the second part of Item 2.
If $u_{0}$ either has finite variance or is radially symmetric, then the solution blows up in finite time. This result was proved in \cite{CG21}.
It remains to consider the case of cylindrically symmetric data.
For $R>1$, we define the radial function $\psi_{R}:\mathbb R^{d-1}\to [0,\infty)$:
\begin{equation}\label{GrindEQ__3_4_}
\psi_{R}(y)=\psi_{R}(r):=R^{2}\theta(r/R),\;r=|y|,
\end{equation}
where $\theta: [0,\infty) \to [0,\infty)$ is given in \eqref{GrindEQ__2_7_}. We define the function on $\mathbb R^{d}$:
\begin{equation}\label{GrindEQ__3_5_}
\phi_{R}(x)=\phi_{R}(y,x_{d}):=\psi_{R}(y)+x_{d}^{2}.
\end{equation}
Using the same argument as in the proof of Lemma 2.7, we have
\begin{eqnarray}\begin{split}\nonumber
V''_{\phi_{R}}(t)&=-\int{\Delta_{y}^{2}\psi_{R}(y)|u(t,x)|^{2}dx}+4\sum_{j,k=1}^{d-1}
{\int{\partial_{jk}^{2}\psi_{R}(y)\textnormal{Re}(\partial_{k}u(t,x)\partial_{j}\bar{u}(t,x))dx}}\\
&~+4c\int{\nabla_{y} \psi_{R}(y)\cdot\frac{y}{|x|^{4}}|u(t,x)|^{2}dx}-\frac{2\sigma}{\sigma+2}
\int{\Delta_{y} \psi_{R}(y)|x|^{-b}|u(t,x)|^{\sigma+2}dx}\\
&~-\frac{4b}{\sigma+2}\int{|y|^{2}\frac{\psi'_{R}(r)}{r}|x|^{-b-2}|u(t,x)|^{\sigma+2}dx}
+8\left\|\partial_{d}u(t)\right\|_{L^2}^{2}-\frac{4\sigma}{\sigma+2}\int{|x|^{-b}|u(t,x)|^{\sigma+2}dx}\\
&~-\frac{8b}{\sigma+2}\int{x_{d}^2|x|^{-b-2}|u(t,x)|^{\sigma+2}dx}+8c\int{x_{d}^2\frac{|u(t,x)|^2}{|x|^4}dx.}\\
\end{split}\end{eqnarray}
We can rewrite it as
\begin{eqnarray}\begin{split}\label{GrindEQ__3_6_}
V''_{\phi_{R}}(t)&=8G(u(t))-8\left\|\nabla_{y}u(t)\right\|_{L^2}^2+\frac{4((d-1)\sigma+2b)}{\sigma+2}P(u(t))\\
&~-\int{\Delta_{y}^{2}\psi_{R}(y)|u(t,x)|^{2}dx}+4\sum_{j,k=1}^{d-1}
{\int{\partial_{jk}^{2}\psi_{R}(y)\textnormal{Re}(\partial_{k}u(t,x)\partial_{j}\bar{u}(t,x))dx}}\\
&~+4c\int{|y|^{2}\frac{\psi'_{R}(r)}{r}\frac{|u(t,x)|^{2}}{|x|^{4}}dx}-\frac{2\sigma}{\sigma+2}
\int{\Delta_{y} \psi_{R}(y)|x|^{-b}|u(t,x)|^{\sigma+2}dx}\\
&~-\frac{4b}{\sigma+2}\int{|y|^{2}\frac{\psi'_{R}(r)}{r}|x|^{-b-2}|u(t,x)|^{\sigma+2}dx}
-\frac{8b}{\sigma+2}\int{x_{d}^2|x|^{-b-2}|u(t,x)|^{\sigma+2}dx}\\
&~+8c\int{x_{d}^2\frac{|u(t,x)|^2}{|x|^4}dx}-8c\int{|x|^{-2}|u(t,x)|^{2}dx}\\
&=8G(u(t))-8\left\|\nabla_{y}u(t)\right\|_{L^2}^2+4\sum_{j,k=1}^{d-1}
{\int{\partial_{jk}^{2}\psi_{R}(y)\textnormal{Re}(\partial_{k}u(t,x)\partial_{j}\bar{u}(t,x))dx}}\\
&~-\int{\Delta_{y}^{2}\psi_{R}(y)|u(t,x)|^{2}dx}+\frac{2\sigma}{\sigma+2}
\int{\left(2(d-1)-\Delta_{y} \psi_{R}(y)\right)|x|^{-b}|u(t,x)|^{\sigma+2}dx}\\
&~+\frac{4b}{\sigma+2}\int{\left(2|x|^2-|y|^{2}\frac{\psi'_{R}(r)}{r}-2x_{d}^2\right)
|x|^{-b-2}|u(t,x)|^{\sigma+2}dx}\\
&~-4c\int{\left(2|x|^2-|y|^{2}\frac{\psi'_{R}(r)}{r}-2x_{d}^2\right)
\frac{|u(t,x)|^2}{|x|^4}dx}.
\end{split}\end{eqnarray}
Using the similar argument as in the proof of \eqref{GrindEQ__2_15_}, we can get
\begin{equation}\label{GrindEQ__3_7_}
-8\left\|\nabla_{y}u(t)\right\|_{L^2}^2+4\sum_{j,k=1}^{d-1}
{\int{\partial_{jk}^{2}\psi_{R}(y)\textnormal{Re}(\partial_{k}u(t,x)\partial_{j}\bar{u}(t,x))dx}}\le 0
\end{equation}
\begin{equation}\label{GrindEQ__3_8_}
\left|\int{\Delta_{y}^{2}\psi_{R}(y)|u(t,x)|^{2}dx}\right|\lesssim R^{-2}.
\end{equation}
Similarly, we also have
\begin{equation}\label{GrindEQ__3_9_}
\left|\int{\left(2(d-1)-\Delta_{y} \psi_{R}(y)\right)|x|^{-b}|u(t,x)|^{\sigma+2}dx}\right|\lesssim \int_{|y|\ge R}{|x|^{-b}|u(t,x)|^{\sigma+2}dx,}
\end{equation}
\begin{equation}\label{GrindEQ__3_10_}
\left|\int{\left(2|x|^2-|y|^{2}\frac{\psi'_{R}(r)}{r}-2x_{d}^2\right)
|x|^{-b-2}|u(t,x)|^{\sigma+2}dx}\right|\lesssim \int_{|y|\ge R}{|x|^{-b}|u(t,x)|^{\sigma+2}dx},
\end{equation}
\begin{equation}\label{GrindEQ__3_11_}
\left|\int{\left(2|x|^2-|y|^{2}\frac{\psi'_{R}(r)}{r}-2x_{d}^2\right)
\frac{|u(t,x)|^2}{|x|^4}dx}\right|\lesssim \int_{|y|\ge R}{|x|^{-2}|u(t,x)|^{2}dx}\lesssim R^{-2}.
\end{equation}
In view of \eqref{GrindEQ__3_6_}--\eqref{GrindEQ__3_11_}, we have
\begin{equation}\label{GrindEQ__3_12_}
V''_{\phi}(t)\le 8G(u(t))+CR^{-2}+CR^{-b}\int_{|y|\ge R}{|u(t,x)|^{\sigma+2}dx.}
\end{equation}
Using the same argument as in the proof of (4.12) in \cite{DK21}, it follows from \eqref{GrindEQ__3_12_} and the fact $\dot{H}_{c}^1\sim\dot{H}^1$ that
\begin{equation}\label{GrindEQ__3_13_}
V''_{\phi}(t)\le 8G(u(t))+CR^{-2}+\left\{\begin{array}{l} {CR^{-\frac{d-2}{2}-b}\left\|u(t)\right\|_{\dot{H}_{c}^1}^{2},~\textrm{if}\;\sigma=2,}
\\{CR^{-\frac{(d-2)\sigma}{4}-b}\left\|u(t)\right\|_{\dot{H}_{c}^1}^{2}
+CR^{-\frac{(d-2)\sigma}{4}-b},~\textrm{if}\;\sigma<2,} \end{array}\right.
\end{equation}
for all $t\in(-T_{*}, T^{*})$. Lemma 3.6 of \cite{CG21} and \eqref{GrindEQ__3_2_} also show that for some $\eta>0$ and $\epsilon>0$,
\begin{equation}\label{GrindEQ__3_14_}
8G(u(t))+\epsilon \left\|u(t)\right\|_{\dot{H}_{c}^1}^2\le -\eta.
\end{equation}
In view of \eqref{GrindEQ__3_13_} and \eqref{GrindEQ__3_14_}, we take $R>1$ sufficiently large to get
$$
V''_{\phi}(t)\le -\frac{\eta}{2}<0,
$$
for all $t\in(-T_{*}, T^{*})$. This shows that $T_{*},T^{*}<\infty$. This completes the proof.
\end{proof}

\subsection{At the ground state threshold}
In this subsection, we study the global existence and blow-up at the ground state threshold.
\begin{lem}[\cite{CG21}]\label{lem 3.1.}
If $d\ge 3$, $c>-c(d)$, $0<b<2$ and $0<\sigma<\frac{4-2b}{d-2}$, then $H_{c}^{1}$ is compactly embedded in $L^{\sigma+2}\left(|x|^{-b}dx\right)$.
\end{lem}
\begin{lem}\label{lem 3.2.}
Let $d\ge 3$, $c>-c(d)$, $0<b<2$ and $0<\sigma<\frac{4-2b}{d-2}$. Let $(f_{n})_{n\ge1}$ be a sequence of $H_{c}^{1}$ function satisfying
$$
M(f_n)=M_{G},~E(f_n)=E_{G},~\forall n\ge 1
$$
and
$$
\lim_{n\to \infty}\left\|f_n\right\|_{\dot{H}_{c}^{1}}=S_{G}.
$$
Then there exists a subsequence still denoted by $(f_{n})_{n\ge1}$ such that
$$
f_{n}\to e^{i\theta}Q~\textrm{strongly in}~H_{c}^{1}~\textrm{as}~n\to \infty
$$
for some $\theta\in\mathbb R$ and some positive solution $Q$ to the ground state equation \eqref{GrindEQ__1_7_}.
\end{lem}
\begin{proof}
Since $(f_{n})_{n\ge 1}$ is a bounded sequence in $H_{c}^{1}$, it follows from Lemma 3.1 that there exists a subsequence still denoted by $(f_{n})_{n\ge 1}$ and a function $f\in H_{c}^{1}$ such that $f_{n}\to f$ weakly in $H_{c}^{1}$ and $P(f_{n})\to P(f)$ as $n\to \infty$. Hence, it follows from \eqref{GrindEQ__2_5_} that
\begin{eqnarray}\begin{split}\nonumber
P(f)=\lim_{n\to \infty}{P(f_{n})}&=\lim_{n\to \infty}{(\sigma+2)\left(\frac{1}{2}\left\|f_{n}\right\|_{\dot{H}_{c}^{1}}^{2}-E(f_{n})\right)}\\
&=(\sigma+2)\left(\frac{1}{2}[S_{G}]^{2}-E_{G}\right)\\
&=\frac{2(\sigma+2)}{d\sigma+2b}[S_{G}]^{2}=P_{G}.
\end{split}\end{eqnarray}
This shows that $f\neq0$. Moreover, using Lemma 2.2, we have
\begin{equation}\label{GrindEQ__3_15_}
P(f)- C_{GN}\left\|f\right\|_{\dot{H}_{c}^{1}}^{\frac{d\sigma+2b}{2}}
\left\|f\right\|_{L^{2}}^{\frac{4-2b-\sigma(d-2)}{2}}\le 0.
\end{equation}
By the lower continuity of weak convergence, we have
$$
\left\|f\right\|_{\dot{H}_{c}^{1}}\le \liminf_{n\to\infty}{\left\|f_{n}\right\|_{\dot{H}_{c}^{1}}}.
$$
Hence, we have
\begin{eqnarray}\begin{split}\label{GrindEQ__3_16_}
P(f)-C_{GN}\left\|f\right\|_{\dot{H}_{c}^{1}}^{\frac{d\sigma+2b}{2}}
\left\|f\right\|_{L^{2}}^{\frac{4-2b-\sigma(d-2)}{2}}
&\ge\liminf_{n\to\infty}{P(f_{n})-C_{GN}\left\|f_{n}\right\|_{\dot{H}_{c}^{1}}^{\frac{d\sigma+2b}{2}}
\left\|f_{n}\right\|_{L^{2}}^{\frac{4-2b-\sigma(d-2)}{2}}}\\
&=P_{G}-C_{GN}[S_{G}]^{\frac{d\sigma+2b}{2}}
[M_{G}]^{\frac{4-2b-\sigma(d-2)}{4}}=0.
\end{split}\end{eqnarray}
In view of \eqref{GrindEQ__3_15_} and \eqref{GrindEQ__3_16_}, we see that $f$ is an optimizer for \eqref{GrindEQ__2_2_}. We also have
$$
\left\|f\right\|_{\dot{H}_{c}^{1}}=\lim_{n\to\infty}{\left\|f_{n}\right\|_{\dot{H}_{c}^{1}}},
$$
which implies that $f_{n}\to f$ strongly in $H_{c}^{1}$. We claim that there exists $\theta\in\mathbb R$ such that $f(x)=e^{i\theta}g(x)$, where $g(x)$ is a non-negative optimizer for \eqref{GrindEQ__2_2_}. Indeed, since $\left|\nabla |f|\right|\le \left|\nabla f\right|$ for all $f\in H_{c}^{1}$, it is clear that $|f|$ is also an optimizer for \eqref{GrindEQ__2_2_} and
\begin{equation}\label{GrindEQ__3_17_}
\left\||f|\right\|_{\dot{H}_{c}^{1}}=\left\|f\right\|_{\dot{H}_{c}^{1}}.
\end{equation}
Using \eqref{GrindEQ__3_17_} and the same argument as in the proof of Lemma 4.2 of \cite{DK21}, the claim follows with $g(x)=|f(x)|$. Since $g$ is an optimizer for \eqref{GrindEQ__2_2_}, Euler--Lagrange equation for $g$ gives
$$
mP_{c}g+ng-\frac{\sigma+2}{C_{GN}}|x|^{-b}|g|^{\sigma}g=0,
$$
where
$$
m=\frac{d\sigma+2b}{2}\left\|f\right\|_{\dot{H}_{c}^{1}}^{\frac{d\sigma+2b-4}{2}}
\left\|f\right\|_{L^{2}}^{\frac{4-2b-(d-2)\sigma}{2}},
$$
$$
n=\frac{4-2b-(d-2)\sigma}{2}\left\|f\right\|_{\dot{H}_{c}^{1}}^{\frac{d\sigma+2b}{2}}
\left\|f\right\|_{L^{2}}^{-\frac{2b+(d-2)\sigma}{2}}.
$$
By a change of variable $g(x)=\lambda Q(\mu x)$ with $\lambda,\mu>0$ satisfying
$$
\mu^{2}=\frac{n}{m}, \lambda^{\sigma}=\frac{nC_{GN}}{\sigma+2}\mu^{-b},
$$
we can see that $Q$ solves \eqref{GrindEQ__1_7_}. As $\left\|g\right\|_{L^2}=\left\|Q\right\|_{L^2}$ and $\left\|g\right\|_{\dot{H}_{c}^1}=\left\|Q\right\|_{\dot{H}_{c}^1}$,
we have $\lambda=\mu=1$. This completes the proof.
\end{proof}
\begin{proof}[\textbf{Proof of Theorem 1.3}] First, we prove Item 1.

\textbf{Proof of Item 1.} Let $u_{0}\in H_{c}^{1}$ satisfy \eqref{GrindEQ__1_19_} and \eqref{GrindEQ__1_20_}. Note that \eqref{GrindEQ__1_19_} and \eqref{GrindEQ__1_20_} are invariant under the scaling
\begin{equation}\label{GrindEQ__3_18_}
u_{0}^{\lambda}(x):=\lambda^{\frac{2-b}{\sigma}}u_{0}(\lambda x),~\lambda>0.
\end{equation}
Hence, by choosing a suitable scaling, we can assume that
\begin{equation}\label{GrindEQ__3_19_}
M(u_{0})=M_{G},~E(u_{0})=E_{G},~\left\|u_{0}\right\|_{\dot{H}_{c}^{1}}<S_{G}.
\end{equation}
We claim that
\begin{equation}\label{GrindEQ__3_20_}
\left\|u(t)\right\|_{\dot{H}_{c}^{1}}<S_{G},
\end{equation}
for all $t\in (-T_{*},T^{*})$. Assume by contradiction that there exists $t_{0}\in (-T_{*},T^{*})$ such that $\left\|u(t_{0})\right\|_{\dot{H}_{c}^{1}}\ge S_{G}$. By continuity argument, there exists $t_{1}\in (-T_{*},T^{*})$ such that  $\left\|u(t_{1})\right\|_{\dot{H}_{c}^{1}}=S_{G}$. It follows from \eqref{GrindEQ__2_5_} and the conservation of energy that
\begin{eqnarray}\begin{split}\nonumber
P(u(t_{1}))&=(\sigma+2)\left(\frac{1}{2}\left\|u(t_{1})\right\|_{\dot{H}_{c}^{1}}^{2}-E(u(t_{1}))\right)\\
&=(\sigma+2)\left(\frac{1}{2}[S_{G}]^{2}-E_{G}\right)\\
&=\frac{2(\sigma+2)}{d\sigma+2b}[S_{G}]^{2}=P_{G}.
\end{split}\end{eqnarray}
Hence, we have
\begin{equation}\nonumber
P(u(t_{1}))-C_{GN}\left\|u(t_{1})\right\|_{\dot{H}_{c}^{1}}^{\frac{d\sigma+2b}{2}}
\left\|u(t_{1})\right\|_{L^{2}}^{\frac{4-2b-\sigma(d-2)}{2}}
=P_{G}-C_{GN}[S_{G}]^{\frac{d\sigma+2b}{2}}
[M_{G}]^{\frac{4-2b-\sigma(d-2)}{4}}=0,
\end{equation}
which shows that $u(t_{1})$ is an optimizer for \eqref{GrindEQ__2_2_}. Using the same argument as in the proof of Lemma 3.2, we have $u(t_{1})=e^{i\theta}Q$ for some $\theta\in \mathbb R$ and some positive solution $Q$ to \eqref{GrindEQ__1_7_}. Moreover, by the uniqueness of solution to \eqref{GrindEQ__1_1_}, we have $u(t)=e^{i(t-t_{1})}e^{i\theta}Q$ which contradicts \eqref{GrindEQ__3_19_}. This shows \eqref{GrindEQ__3_20_} and the solution
exists globally in time. We divide the study in two cases.\\
\textbf{Case 1.1.} Let us consider the case
$$
\sup_{t\in\mathbb R}{\left\|u(t)\right\|_{\dot{H}_{c}^{1}}}<S_{G}.
$$
There exists $\rho>0$ such that
\begin{equation}\label{GrindEQ__3_21_}
\left\|u(t)\right\|_{\dot{H}_{c}^{1}}\le (1-\rho)S_{G},
\end{equation}
for all $t\in\mathbb R$. In view of \eqref{GrindEQ__3_19_} and \eqref{GrindEQ__3_21_}, we have \eqref{GrindEQ__3_2_}.
Repeating the same argument as in the proof of Item 1 of Theorem 1.2, we can prove \eqref{GrindEQ__1_21_}.\\
\textbf{Case 1.2.} Let us consider the case
$$
\sup_{t\in\mathbb R}{\left\|u(t)\right\|_{\dot{H}_{c}^{1}}}=S_{G}.
$$ Then there exists
a time sequence $(t_{n})_{n\ge1}$ such that
\begin{equation}\nonumber
M(u(t_{n}))=M_{G},~E(u(t_{n}))=E_{G},~
\lim_{n\to\infty}\left\|u(t_{n})\right\|_{\dot{H}_{c}^{1}}=S_{G}.
\end{equation}
We claim that $|t_{n}|\to \infty$. Assume by contradiction that there exists subsequence still denoted by $(t_{n})_{n\ge1}$ and $t_{0}\in \mathbb R$ such that $t_{n}\to t_{0}$ as $n\to \infty$. By continuity of solution, $u(t_{n})\to u(t_{0})$ strongly in $H^{1}_{c}$.
This shows that $\left\|u(t_{0})\right\|_{\dot{H}_{c}^{1}}=S_{G}$. Using the same argument above, we see that $u(t_{0})$ is an optimizer for \eqref{GrindEQ__2_2_} which is a contradiction.
Hence, it follows from Lemma 3.2 that there exists a subsequence still denoted by $(t_{n})_{n\ge1}$ such that
$$
u(t_{n})\to e^{i\theta}Q~\textrm{strongly in}~H_{c}^{1}~\textrm{as}~n\to \infty,
$$
for some $\theta\in\mathbb R$ and some positive solution $Q$ to  \eqref{GrindEQ__1_7_}.

\textbf{Proof of Item 2.} Let $u_{0}$ satisfy \eqref{GrindEQ__1_19_} and \eqref{GrindEQ__1_23_}. By scaling \eqref{GrindEQ__3_18_}, we can assume that
\begin{equation}\nonumber
M(u_{0})=M_{G},~E(u_{0})=E_{G},~\left\|u_{0}\right\|_{\dot{H}_{c}^{1}}=S_{G}.
\end{equation}
This shows that $u_{0}$ is optimizer for \eqref{GrindEQ__2_2_} which implies that $u_{0}=e^{i\theta}Q(x)$ for some $\theta\in \mathbb R$ and some positive solution $Q$ to \eqref{GrindEQ__1_7_}. By uniqueness of solution to \eqref{GrindEQ__1_1_}, we have $u(t)=e^{it}e^{i\theta}Q(x)$.

\textbf{Proof of Item 3.} Let $u_{0}$ satisfy \eqref{GrindEQ__1_19_} and \eqref{GrindEQ__1_24_}. By scaling (3.18), we can assume that
\begin{equation}\label{GrindEQ__3_22_}
M(u_{0})=M_{G},~E(u_{0})=E_{G},~\left\|u_{0}\right\|_{\dot{H}_{c}^{1}}>S_{G}.
\end{equation}
Using the same argument as in the proof of Item 1, we can prove that
\begin{equation}\nonumber
\left\|u(t)\right\|_{\dot{H}_{c}^{1}}>S_{G},
\end{equation}
for all $t\in (-T_{*},T^{*})$. We only consider the positive times. The negative time is treated similarly. If $T^{*}<\infty$, then we are done. If $T^{*}=\infty$, we have two possible cases.\\
\textbf{Case 3.1.} Let us consider the case
$$
\sup_{t\in[0,\infty)}{\left\|u(t)\right\|_{\dot{H}_{c}^{1}}}>S_{G}.
$$
Then there exists $\delta>0$ such that
\begin{equation}\label{GrindEQ__3_23_}
\left\|u(t)\right\|_{\dot{H}_{c}^{1}}\ge (1+\delta)S_{G}
\end{equation}
for all $t\in[0,\infty)$.
It follows from \eqref{GrindEQ__2_6_}, \eqref{GrindEQ__3_22_}, \eqref{GrindEQ__3_23_} and the conservation of mass and energy, we have
\begin{eqnarray}\begin{split}\nonumber
G(u(t))[M(u(t))]^{\sigma_{c}}&=\frac{d\sigma+2b}{2}E(u(t))[M(u(t))]^{\sigma_{c}}-\frac{d\sigma-4+2b}{4}
\left(\left\|u(t)\right\|_{\dot{H}_{c}^{1}}\left\|u(t)\right\|_{L^{2}}^{\sigma_{c}}\right)^2\\
&\le \frac{d\sigma+2b}{2}H_{G}-\frac{d\sigma-4+2b}{4}
\left((1+\delta)K_{G}\right)^2\\
&=-\frac{d\sigma-4+2b}{4}\left((1+\delta)^2-1\right)[K_{G}]^2,
\end{split}\end{eqnarray}
for all $t\in [0,\infty)$. By Theorem 1.1, there exists a time sequence $t_{n}\to \infty$ such that $\left\|u(t_{n})\right\|_{\dot{H}_{c}^{1}}\to \infty$ as $n\to \infty$.\\
\textbf{Case 3.2.} Let us consider the case
$$
\sup_{t\in[0,\infty)}{\left\|u(t)\right\|_{\dot{H}_{c}^{1}}}=S_{G}.
$$
Then there exists a time sequence $(t_{n})_{n\ge 1}$ such that $\left\|u(t_{n})\right\|_{\dot{H}_{c}^{1}}\to S_{G}$ as $n\to\infty$. Using the same argument as in Case 1.2, we can see that there exists a subsequence still denoted by $(t_{n})_{n\ge1}$ such that $t_{n}\to \infty$ and
$$
u(t_{n})\to e^{i\theta}Q~\textrm{strongly in}~H_{c}^{1}~\textrm{as}~n\to \infty
$$
for some $\theta\in\mathbb R$ and some positive solution $Q$ to  \eqref{GrindEQ__1_7_}. This completes the proof of the first part of Item 3. Let us prove the second part.
\begin{itemize}
\item \textbf{Finite variance data.} If $|x|u_{0}\in L^{2}$, then Case 3.1 is excluded. Indeed, if Case 3.1 occurs, we have
      $$
      G(u(t))<-\eta,~\forall t\in [0,\infty)
      $$
      for some $\eta>0$. Then Lemma 2.6 shows that
      $$
      \frac{d^2}{dt^2}\left\|xu(t)\right\|_{L^{2}}^{2}=8G(u(t))<-8\eta<0,
      $$
      which contradicts $T^{*}=\infty$.
\item \textbf{Radially symmetric data.} If Case 3.1 occurs, we have \eqref{GrindEQ__3_23_}. It follows from \eqref{GrindEQ__2_6_}, \eqref{GrindEQ__3_22_}, \eqref{GrindEQ__3_23_} and the conservation of mass and energy that
      \begin{eqnarray}\begin{split}\nonumber
      &[M(u(t))]^{\sigma_{c}}\left(8G(u(t))+\varepsilon\left\|u(t)\right\|_{\dot{H}_{c}^1}\right)\\
      &~~~~~~~~~~=4(d\sigma+2b)E(u(t))[M(u(t))]^{\sigma_{c}}
      -(2d\sigma-8+4b-\varepsilon)   \left\|u(t)\right\|_{\dot{H}_{c}^{1}}^{2}[M(u(t))]^{\sigma_{c}}\\
      &~~~~~~~~~~\le 4(d\sigma+2b)H_{G}-(2d\sigma-8+4b-\varepsilon)
      \left((1+\delta)K_{G}\right)^2\\
      &~~~~~~~~~~=-2(d\sigma-4+2b)
      \left((1+\delta)K_{G}\right)^2
      \left[\frac{(1+\delta)^2-1}{(1+\delta)^2}-\frac{\varepsilon}{2(d\sigma-4+2b)}\right],
      \end{split}\end{eqnarray}
      for all $t\in [0,\infty)$. Taking $\varepsilon>0$ sufficiently small, there exists $\eta>0$ such that
      \begin{equation}\label{GrindEQ__3_24_}
      8G(u(t))+\varepsilon \left\|u(t)\right\|_{\dot{H}_{c}^{1}}\le -\eta,
      \end{equation}
      for all $t\in [0,\infty)$. On the other hand, Lemma 4.6 of \cite{JAK21} shows that
      \begin{equation} \label{GrindEQ__3_25_}
      V''_{\varphi_{R}}(t)\le 8 G(u(t))+O\left(R^{-2}+\varepsilon^{-\frac{\sigma}{4-\sigma}}R^{-{\frac{2\left[(d-1)\sigma+2b\right]}
      {4-\sigma}}}+\varepsilon\left\|u(t)\right\| _{\dot{H}^{1}_{c}}^{2}\right),
      \end{equation}
      for any $\varepsilon>0$ and any $R>1$. Taking $\varepsilon>0$ small enough and $R>1$ large enough, it follows from \eqref{GrindEQ__3_23_} and \eqref{GrindEQ__3_24_} that
      $$
      V''_{\varphi_{R}}(t)\le -\frac{\eta}{2},
      $$
      for all $t\in [0,\infty)$. This is also impossible.

  \item \textbf{Cylindrically symmetric data.} Using \eqref{GrindEQ__3_13_}, \eqref{GrindEQ__3_24_} and repeating the same argument above, we can easily prove that Case 3.1 is excluded. This completes the proof.
  \end{itemize}
\end{proof}
\subsection{Above the ground state threshold}
In this subsection, we study the global existence and blow-up above the ground state threshold.
\begin{lem}[\cite{W83}]\label{lem 3.3.}
Let $u\in H^{1}$ and $|x|u\in L^{2}$. Then we have
\begin{equation}\nonumber
\left\|u\right\|_{2}^{2}\le\frac{2}{d}\left\||x|u\right\|_{2}\left\|\nabla u\right\|_{2}.
\end{equation}
\end{lem}
\begin{lem}\label{lem 3.4.}
Let $d\ge 3$, $0<b<2$, $\frac{4-2b}{d}<\sigma<\frac{4-2b}{d-2}$ and $c>-c(d)$. Assume $u\in H_{c}^{1}$ and $|x|u\in L^{2}$. Then we have
\begin{equation}\label{GrindEQ__3_26_}
\left[\textnormal{Im}\int{x\cdot\bar{u}\nabla udx}\right]^{2}\le\int{|x|^{2}|u|^{2}dx}\left[\left\| u\right\|_{\dot{H}_{c}^{1}}^{2}-\frac{[P(u)]^{\frac{4}{d\sigma+2b}}}
{C_{GN}^{\frac{4}{d\sigma+2b}}[M(u)]^{\frac{4-2b-(d-2)\sigma}{d\sigma+2b}}}\right].
\end{equation}
\end{lem}
\begin{proof}
Noticing
$$
\int{|\nabla(e^{i\lambda|x|^{2}}u)|^{2}dx}=4\lambda^{2}\int{|x|^{2}|u|^{2}dx}+4\lambda
\textnormal{Im}\int{x\cdot\bar{u}\nabla udx}+\int{|\nabla u|^{2}dx},
$$
we have
\begin{equation}\label{GrindEQ__3_27_}
\left\|e^{i\lambda|x|^{2}}u\right\|_{\dot{H}_{c}^{1}}^{2}=4\lambda^{2}\int{|x|^{2}|u|^{2}dx}+4\lambda
\textnormal{Im}\int{x\cdot\bar{u}\nabla udx}+\left\|u\right\|_{\dot{H}_{c}^{1}}^{2}.
\end{equation}
On the other hand, Lemma 2.2 (sharp Gagliardo-Nirenberg inequality) implies that
\begin{equation}\label{GrindEQ__3_28_}
[P(u)]^{\frac{4}{d\sigma+2b}}=[P(e^{i\lambda|x|^{2}}u)]^{\frac{4}{d\sigma+2b}}\le C_{GN}^{\frac{4}{d\sigma+2b}}\left\|e^{i\lambda|x|^{2}}u\right\|_{\dot{H}_{c}^{1}}^{2}
\left\|u\right\|_{L^{2}}^{\frac{2[4-2b-(d-2)\sigma]}{d\sigma+2b}}
\end{equation}
It follows from \eqref{GrindEQ__3_27_} and \eqref{GrindEQ__3_28_} that
$$
4\lambda^2\int{|x|^{2}|u|^{2}dx}+4\lambda\textnormal{Im}\int{x\cdot\bar{u}\nabla{u}dx}
+\left\|u\right\|_{\dot{H}_{c}^{1}}^{2}
\ge\frac{[P(u)]^{\frac{4}{d\sigma+2b}}}
{C_{GN}^{\frac{4}{d\sigma+2b}}[M(u)]^{\frac{4-2b-(d-2)\sigma}{d\sigma+2b}}},
$$
for $\lambda\in\mathbb R$, where the left-hand side is a polynomial in $\lambda$. The discriminant of this polynomial in $\lambda$ must be negative or null, which yields the conclusion of this lemma.
\end{proof}
\begin{proof}[\textbf{Proof of Theorem 1.4.}] Let $u_{0}\in H_{c}^{1}$ be such that $|x|u_{0}\in L^2$. Assume further that $u_{0}$ satisfies \eqref{GrindEQ__1_26_} and \eqref{GrindEQ__1_27_}. Using Lemma 2.6, we have
\begin{equation}\label{GrindEQ__3_29_}
\left\|u(t)\right\|_{\dot{H}_{c}^{1}}^{2}=\frac{4(d\sigma+2b)E(u(t))-V''(t)}{2(d\sigma+2b-4)},
\end{equation}
\begin{equation}\label{GrindEQ__3_30_}
P(u(t))=\frac{(16E(u(t))-V''(t))(\sigma+2)}{4(d\sigma+2b-4)}.
\end{equation}
\eqref{GrindEQ__3_30_} yields that $V''(t)\le 16 E(u(t))=16 E(u_{0})$ for all $t$. Using Lemma 2.4, Lemma 3.4, \eqref{GrindEQ__3_29_}, \eqref{GrindEQ__3_30_} and the conservations of mass and energy, we immediately have
\begin{equation}\label{GrindEQ__3_31_}
(V'(t))^{2}\le 16 V(t)\left[\frac{4(d\sigma+2b)E(u_{0})-V''(t)}{2d\sigma+4b-8}-
\frac{\left(\frac{(\sigma+2)(16E(u_{0})-V''(t))}{4d\sigma+8b-16}\right)^{\frac{4}{d\sigma+2b}}}
{C_{GN}^{\frac{4}{d\sigma+2b}}[M(u_{0})]^{\frac{4-2b-(d-2)\sigma}{d\sigma+2b}}}
\right].
\end{equation}
Putting $z(t):=\sqrt{V(t)}$, it follows from \eqref{GrindEQ__3_31_} that
\begin{equation}\label{GrindEQ__3_32_}
(z'(t))^{2}\le4f(V''(t)),
\end{equation}
where
\begin{equation}\label{GrindEQ__3_33_}
f(x)=-\frac{x}{2d\sigma+4b-8}+\frac{(2d\sigma+4b)E(u_{0})}{d\sigma+2b-4}
-\frac{\left(\frac{(\sigma+2)(16E(u_{0})-x)}{4d\sigma+8b-16}\right)^{\frac{4}{d\sigma+2b}}}
{C_{GN}^{\frac{4}{d\sigma+2b}}[M(u_{0})]^{\frac{4-2b-(d-2)\sigma}{d\sigma+2b}}}
\end{equation}
with $x\le 16E(u_{0})$.
An easy computation shows that
$$
f'(x)=-\frac{1}{2d\sigma+4b-8}
+\frac{\frac{4}{d\sigma+2b}\left(\frac{\sigma+2}{4d\sigma+8b-16}\right)^{\frac{4}{d\sigma+2b}}
\left(16E(u_{0})-x\right)^{\frac{4-d\sigma-2b}{d\sigma+2b}}}
{C_{GN}^{\frac{4}{d\sigma+2b}}[M(u_{0})]^{\frac{4-2b-(d-2)\sigma}{d\sigma+2b}}}.
$$
Since $\sigma>\frac{4-2b}{d}$, $f(x)$ is decreasing on $(-\infty,x_{0})$, increasing on $(x_{0},16E(u_{0})]$, where $x_{0}$ satisfies
\begin{equation}\label{GrindEQ__3_34_}
\frac{1}{2d\sigma+4b-8}=
\frac{\frac{4}{d\sigma+2b}\left(\frac{\sigma+2}{4d\sigma+8b-16}\right)^{\frac{4}{d\sigma+2b}}
\left(16E(u_{0})-x_{0}\right)^{\frac{4-d\sigma-2b}{d\sigma+2b}}}
{C_{GN}^{\frac{4}{d\sigma+2b}}[M(u_{0})]^{\frac{4-2b-(d-2)\sigma}{d\sigma+2b}}}.
\end{equation}
In view of \eqref{GrindEQ__3_33_} and \eqref{GrindEQ__3_34_}, we also have $f(x_{0})=\frac{x_{0}}{8}$.
Using \eqref{GrindEQ__2_4_}, \eqref{GrindEQ__2_6_} and \eqref{GrindEQ__3_34_}, we also have
\begin{equation}\label{GrindEQ__3_35_}
\left[\frac{M_{G}}{M(u_{0})}\right]^{\sigma_{c}}
\frac{16E_{G}}{16E(u_{0})-x_{0}}=1.
\end{equation}
Hence, \eqref{GrindEQ__1_26_} is equivalent to
\begin{equation}\label{GrindEQ__3_36_}
x_{0}\ge0.
\end{equation}
Noticing
\begin{equation}\nonumber
(z'(t))^{2}=\frac{(V'(t))^{2}}{4V(t)},
\end{equation}
and using \eqref{GrindEQ__3_35_}, we can see that \eqref{GrindEQ__1_27_} is equivalent to
\begin{equation}\label{GrindEQ__3_37_}
(z'(0))^{2}\ge4f(x_{0})=\frac{x_{0}}{2}.
\end{equation}

\textbf{Proof of Item 1.} Let $u_{0}$ satisfy \eqref{GrindEQ__1_28_} and \eqref{GrindEQ__1_29_}. The assumption \eqref{GrindEQ__1_29_} means that
\begin{equation}\label{GrindEQ__3_38_}
z'(0)\ge0.
\end{equation}
Using \eqref{GrindEQ__2_5_} and \eqref{GrindEQ__3_35_}, we can see that the assumption \eqref{GrindEQ__1_28_} is equivalent to
\begin{equation}\label{GrindEQ__3_39_}
\left[\frac{M_{G}}{M(u_{0})}\right]^{\sigma_{c}}
\frac{4(\sigma+2)E_{G}}
{(d\sigma+2b-4)P(u_{0})}>1=
\left[\frac{M_{G}}{M(u_{0})}\right]^{\sigma_{c}}
\frac{16E_{G}}{16E(u_{0})-x_{0}}.
\end{equation}
Using \eqref{GrindEQ__3_30_}, we can see that \eqref{GrindEQ__3_39_} is equivalent to
\begin{equation}\label{GrindEQ__3_40_}
V''(0)>x_{0}.
\end{equation}
By \eqref{GrindEQ__3_40_}, we can take $\delta_{1}>0$ such that
$$
V''(0)\ge x_{0}+2\delta_{1}.
$$
By continuity argument, we have
\begin{equation}\label{GrindEQ__3_41_}
V''(t)> x_{0}+\delta_{1},~\forall t\in [0,t_{1}),
\end{equation}
for all $t_{1}>0$ sufficiently small.
By reducing $t_{1}$ if necessary, we can assume that
\begin{equation}\label{GrindEQ__3_42_}
z'(t_{1})>2\sqrt{f(x_{0})}.
\end{equation}
In fact, using \eqref{GrindEQ__3_37_} and \eqref{GrindEQ__3_38_}, we have $z'(0)\ge 2\sqrt{f(x_{0})}$. If the inequality is strict, then \eqref{GrindEQ__3_42_} follows from the continuity argument. If $z'(0)=2\sqrt{f(x_{0})}$, it follows from \eqref{GrindEQ__3_40_} that
$$
z''(0)=\frac{1}{z(0)}\left[\frac{V''(0)}{2}-(z'(0))^{2}\right]>0.
$$
Hence, we can take $t_{1}>0$ small enough such that \eqref{GrindEQ__3_42_} holds. We then choose $\varepsilon \in (0,1)$ small enough such that
\begin{equation}\nonumber
z'(t_{1})\ge 2\sqrt{f(x_{0})}+2\varepsilon.
\end{equation}
We will prove by contradiction that
\begin{equation}\label{GrindEQ__3_43_}
z'(t)> 2\sqrt{f(x_{0})}+\varepsilon,  ~ \forall t\ge t_{1}.
\end{equation}
If \eqref{GrindEQ__3_43_} does not hold, there exists $t_{2}$ satisfying
$$
t_{2}=\inf\left\{t\ge t_{1}:~z'(t)\le 2\sqrt{f(x_{0})}+\varepsilon\right\}.
$$
By continuity of $z'(t)$, we have
\begin{equation}\label{GrindEQ__3_44_}
z'(t_{2})= 2\sqrt{f(x_{0})}+\varepsilon,
\end{equation}
and
\begin{equation}\label{GrindEQ__3_45_}
z'(t)\ge 2\sqrt{f(x_{0})}+\varepsilon,~~\forall t\in[t_{1},t_{2}].
\end{equation}
Using \eqref{GrindEQ__3_45_} and \eqref{GrindEQ__3_32_}, we immediately have
$$
(2\sqrt{f(x_{0})}+\varepsilon)^{2}\le(z'(t))^{2}\le 4f(V''(t)),~~\forall t\in[t_{1},t_{2}],
$$
which implies that $f(V''(t))>f(x_{0})$ for all $t\in [t_{1},t_{2}]$. This means that $V''(t)\neq x_{0}$ for all $t\in [t_{1},t_{2}]$. Using \eqref{GrindEQ__3_40_} and the continuity of $V''(t)$, we have $V''(t)>x_{0}$ for all $t\in [t_{1},t_{2}]$.
We claim that there exists a constant $C>0$ (which is independent of $\varepsilon$ and $t$) such that
\begin{equation}\label{GrindEQ__3_46_}
V''(t)\ge x_{0}+\frac{\sqrt{\varepsilon}}{C},~~\forall t\in [t_{1},t_{2}].
\end{equation}
In fact, by the Taylor expansion of $f$ around $x=x_{0}$, we get
\begin{equation}\label{GrindEQ__3_47_}
x_{0}<x<x_{0}+\frac{16E(u_{0})-x_{0}}{2} \Rightarrow f(x)\le f(x_{0})+\frac{M}{2} (x-x_{0})^{2},
\end{equation}
where $M:=\max\left\{f''(x):x\in \left[x_{0}, \frac{16E(u_{0})+x_{0}}{2}\right]\right\}$. It is obvious that $M>0$. If $V''(t)\ge x_{0}+\frac{16E(u_{0})-x_{0}}{2}$, then we take $C=C_{1}:=\frac{2}{16E(u_{0})-x_{0}}$. If $x_{0}<V''(t)<x_{0}+\frac{16E(u_{0})-x_{0}}{2}$, then it follows from \eqref{GrindEQ__3_32_}, \eqref{GrindEQ__3_45_} and \eqref{GrindEQ__3_47_} that
$$
(2\sqrt{f(x_{0})}+\varepsilon)^{2}\le(z'(t))^{2}\le 4f(V''(t))\le 4f(x_{0})+2M(V''(t)-x_{0}),
$$
which implies that
$$
4\sqrt{f(x_{0})}\varepsilon+\varepsilon^{2}<2M(V''(t)-x_{0})^{2}.
$$
Hence, we can take $C=C_{2}:=\sqrt{M}(4f(x_{0}))^{-\frac{1}{4}}$. Therefore, taking $C:=\max\left\{C_{1},C_{2}\right\}$, we have \eqref{GrindEQ__3_46_} for all $t\in [t_{1},t_{2}]$.
In view of \eqref{GrindEQ__3_44_} and \eqref{GrindEQ__3_46_}, we have
\begin{eqnarray}\begin{split}\label{GrindEQ__3_48_}
z''(t_{2})&=\frac{1}{z(t_{2})}\left[\frac{V''(t_{2})}{2}-(z'(t_{2}))^{2}\right]\\
&\ge\frac{1}{z(t_{2})}\left[\frac{x_{0}}{2}+\frac{\sqrt{\varepsilon}}{2C}
-(2\sqrt{f(x_{0})}+\varepsilon)^{2}\right]\\
&\ge \frac{1}{z(t_{2})}\left[\frac{\sqrt{\varepsilon}}{2C}
-4\varepsilon\sqrt{f(x_{0})}-\varepsilon^{2}\right].
\end{split}\end{eqnarray}
Taking $\varepsilon >0$ small enough in \eqref{GrindEQ__3_48_}, we have $z''(t_{2})>0$, which is contrary to \eqref{GrindEQ__3_44_} and \eqref{GrindEQ__3_45_}. This completes the proof of \eqref{GrindEQ__3_43_}. Using \eqref{GrindEQ__3_43_} and the same argument as in the proof of \eqref{GrindEQ__3_46_}, we can prove
\begin{equation}\label{GrindEQ__3_49_}
V''(t)\ge x_{0}+\frac{\sqrt{\varepsilon}}{C},~~\forall t\ge t_{1}.
\end{equation}
In view of \eqref{GrindEQ__3_41_} and \eqref{GrindEQ__3_49_}, we have
\begin{equation}\label{GrindEQ__3_50_}
V''(t)\ge x_{0}+\delta,~~\forall t\ge 0,
\end{equation}
where $\delta=\min\left\{\delta_{1}, \frac{\sqrt{\varepsilon}}{C}\right\}>0$.
Using \eqref{GrindEQ__3_30_}, \eqref{GrindEQ__3_50_}, \eqref{GrindEQ__3_35_},  \eqref{GrindEQ__2_5_} and the conservations of mass and energy, we have for any $t>0$
\begin{eqnarray}\begin{split}\label{GrindEQ__3_51_}
P(u(t))[M(u(t))]^{\sigma_{c}}&=
[M(u_{0})]^{\sigma_{c}}\frac{(16E(u_{0})-V''(t))(\sigma+2)}{4(d\sigma+2b)-16}\\
&\le [M(u_{0})]^{\sigma_{c}}\frac{(16E(u_{0})-x_{0}-\delta)(\sigma+2)}
{4(d\sigma+2b)-16}\\
&=\frac{4(\sigma+2)}{d\sigma-4+2b}[M_{G}]^{\sigma_{c}}E_{G}-\frac{\sigma+2}{4(d\sigma-4+2b)}\delta [M(u_{0})]^{\sigma_{c}}\\
&=(1-\rho)P_{G}[M_{G}]^{\sigma_{c}},
\end{split}\end{eqnarray}
where
$$
\rho:=\frac{\sigma+2}{4(d\sigma-4+2b)}\frac{[M(u_{0})]^{\sigma_{c}}}{P_{G}[M_{G}]^{\sigma_{c}}}\delta>0.
$$
This shows \eqref{GrindEQ__1_9_}, and thus the solution exists globally in time.

\textbf{Proof of Item 2.} Let $u_{0}$ satisfy \eqref{GrindEQ__1_30_} and \eqref{GrindEQ__1_31_}.
Using the same argument as in the proof of \eqref{GrindEQ__3_38_} and \eqref{GrindEQ__3_40_}, we can prove that the assumptions \eqref{GrindEQ__1_31_} and \eqref{GrindEQ__1_30_} are respectively equivalent to
\begin{equation}\label{GrindEQ__3_52_}
z'(0)\le0,
\end{equation}
\begin{equation}\label{GrindEQ__3_53_}
V''(0)<x_{0}.
\end{equation}
We will prove by contradiction that
\begin{equation}\label{GrindEQ__3_54_}
z''(t)<0,
\end{equation}
for any $t$ in the existence time. In fact, it follows from \eqref{GrindEQ__3_37_} and \eqref{GrindEQ__3_53_} that
\begin{equation}\label{GrindEQ__3_55_}
z''(0)=\frac{1}{z(0)}\left[\frac{V''(0)}{2}-(z'(0))^{2}\right]
<\frac{1}{z(0)}\left(\frac{x_{0}}{2}-\frac{x_{0}}{2}\right)=0.
\end{equation}
Assume that \eqref{GrindEQ__3_54_} does not hold. Then there exists $t_{0}>0$ in the existence time such that
$$
z''(t_{0})=0~~\textnormal{and}~~z''(t)<0,~~\forall t\in[0,t_{0}).
$$
In view of \eqref{GrindEQ__3_37_} and \eqref{GrindEQ__3_52_}, we have
\begin{equation}\label{GrindEQ__3_56_}
z'(t)<z'(0)\le -\sqrt{4f(x_{0})},~~\forall t\in (0,t_{0}],
\end{equation}
which implies that $(z'(t))^{2}>4f(x_{0})$. Hence, using \eqref{GrindEQ__3_32_}, we immediately have
\begin{equation}\label{GrindEQ__3_57_}
f(V''(t))>f(x_{0}),~~\forall t\in (0,t_{0}].
\end{equation}
Using \eqref{GrindEQ__3_57_}, \eqref{GrindEQ__3_53_} and the continuity of $V''(t)$, we can see that
\begin{equation}\label{GrindEQ__3_58_}
V''(t)<x_{0},~~\forall t\in [0,t_{0}].
\end{equation}
In view of \eqref{GrindEQ__3_56_} and \eqref{GrindEQ__3_58_}, we have
\begin{equation}\nonumber
z''(t_{0})=\frac{1}{z(t_{0})}\left[\frac{V''(t_{0})}{2}-(z'(t_{0}))^{2}\right]
<\frac{1}{z(t_{0})}\left(\frac{x_{0}}{2}-\frac{x_{0}}{2}\right)=0,
\end{equation}
which contradicts $z''(t_{0})=0.$ Hence, we have \eqref{GrindEQ__3_54_} for any $t$ in the existence time. Using \eqref{GrindEQ__3_52_} and \eqref{GrindEQ__3_54_} and the fact $z(t)=\sqrt{V(t)}$, we can see that there exists $T<\infty$ such that
$\lim_{t\rightarrow T}{V(t)}=0$.
Noticing $\dot{H}_{c}^{1}\sim \dot{H}^{1}$, it follows from Lemma 3.3 that $\lim_{t\rightarrow T}{\left\|u(t)\right\|_{\dot{H}_{c}^{1}}}=\infty.$ This completes the proof.
\end{proof}
%%%%%%%%%%%%%%%%%%%%%%%%%%%%%%%%%%%%%%%%%%%%%%%%%%%%%%%%%%%%%%%%%%%%%%%%%%%%%
%\section*{References}


\begin{thebibliography}{26}
\bibitem{AK21} \small{J. An and J. Kim, Local well-posedness for the inhomogeneous nonlinear Schr\"{o}dinger equation in $H^{s}(\mathbb R^{n} )$, \textit{Nonlinear Anal. Real World Appl.}, \textbf{59} (2021), 103268.}
\bibitem{AKC21} \small{J. An, J. Kim and K. Chae, Continuous dependence of the Cauchy problem for the inhomogeneous nonlinear Schr\"{o}dinger equation in $H^{s}(\mathbb R^{n} )$, \textit{Discrete Contin. Dyn. Syst. Ser. B}, (2021), doi:10.3934/dcdsb.2021221.}
\bibitem{AC21} \small{A. H. Ardila and M. Cardoso, Blow-up solutions and strong instability of ground states for the inhomogeneous nonlinear Schr\"{o}dinger equation, \textit{Commun. Pure Appl. Anal.}, \textbf{20}(1) (2021), 101--119.}
\bibitem{BPVT07} \small{J. Belmonte-Beitia, V. M. P\'{e}rez-Garc\'{i}a, V. Vekslerchik and P. J. Torres, Lie symmetries and solitons in nonlinear systems with spatially inhomogeneous nonlinearities, \textit{Phys. Rev. Lett.}, \textbf{98}(6) (2007), 064102.}
\bibitem{BPST03} \small{N. Burq, F. Planchon, J. Stalker and A. S. Tahvildar-Zadeh, Strichartz estimates for the wave and Schr\"{o}dinger equations with the inverse-square potential, \textit{J. Funct. Anal.}, \textbf{203} (2003), 519--549.}
\bibitem{C21} \small{L. Campos, Scattering of radial solutions to the inhomogeneous nonlinear Schr\"{o}dinger equation, \textit{Nonlinear Anal.}, \textbf{202} (2021) 112118.}
\bibitem{CG21} \small{L. Campos and C. M. Guzm\'{a}n, On the inhomogeneous NLS with inverse-square potential, \textit{Z. Angew. Math. Phys.}, \textbf{72:143} (2021), Avaliable at  https://doi.org/10.1007/s0033-021-01560-4.}
\bibitem{C03} \small{T. Cazenave, \textit{Semilinear Schr\"{o}dinger Equations}, Courant Lecture Notes in Mathematics, New York University, Courant Institute of Mathematical Sciences, New York; American Mathematical Society, Providence, RI, 2003.}
\bibitem{D18} \small{V. D. Dinh, Global exsitence and blowup for a class of focusing nonlinear Schr\"{o}dinger equation with inverse-square potential, \textit{J. Math. Anal. Appl.}, \textbf{468} (2018), 270--303.}
\bibitem{DK21} \small{V. D. Dinh and S. Keraani, Long time dynamics of non-radial solutions to inhomogeneous nonlinear Schr\"{o}dinger equations, \textit{SIAM J. Math. Anal.}, \textbf{53}(4) (2021), 4765--4811.}
\bibitem{G77} \small{R. T. Glassey, On the blowing up of solutions to the Cauchy problem for nonlinear Schr\"{o}dinger equations, \textit{J. Math. Phys.}, \textbf{18} (1977), 1794--1797.}
\bibitem{JAK21} \small{R. Jang, J. An and J. Kim, The Cauchy problem for the energy-critical inhomogeneous nonlinear Schr\"{o}dinger equation with inverse--square potential, Preprint arXiv:2107.09826.}
\bibitem{KSWW75} \small{H. Kalf, U. W. Schmincke, J. Walter and R. Wust, On the spectral theory of Schr\"{o}dinger and Dirac operators with strongly singular potentials, in: Spectral Theory and Differential Equations, in: Lect. Notes in Math., vol. 448, Springer, Berlin, 1975, pp. 182--226.}
\bibitem{KMVBT17} \small{Y. V. Kartashov, B. A. Malomed, V. A. Vysloukh, M. R. Belic and L. Torner, Rotating vortex clusters in media with inhomogeneous defocusing nonlinearity. \textit{Opt. Lett.}, \textbf{42}(3) (2017), 446--449.}
\bibitem{KMVZZ17} \small{R. Killip, C. Miao, M. Visan, J. Zhang and J. Zheng, The energy-critical NLS with inverse-square potential, \textit{Discrete Contin. Dyn. Syst.}, \textbf{37} (2017), 3831--3866.}
\bibitem{KMVZ17} \small{R. Killip, J. Murphy, M. Visan and J. Zheng, The focusing cubic NLS with inverse-square potential in three space dimensions, \textit{Differential Integral Equations}, \textbf{30}(3--4) (2017), 161--206.}
\bibitem{LP15} \small{F. Linares and G. Ponce, \textit{Introduction to Nonlinear Dispersive Equations}, second edition, Universitext. Springer, New York, 2015.}
\bibitem{MM18} \small{J. Lu, C. Miao and J. Murphy, Scattering in $H^{1}$ for the intercritical NLS with an inverse-square potential, \textit{J. Differential Equations} \textbf{264}(5) (2018), 3174--3211.}
\bibitem{S16} \small{T. Suzuki. Solvability of nonlinear Schr\"{o}dinger equations with some critical singular potential via generalized Hardy-Rellich inequalities, \textit{Funkcial. Ekvac.}, \textbf{59}(1) (2016), 1--34.}
\bibitem{WHHG11} \small{B. X. Wang, Z. Huo, C. Hao and Z. Guo, \textit{Harmonic Analysis Method for Nonlinear Evolution Equations, I}, World Scientific Publishing Co. Pte. Ltd., Hackensack, NJ, 2011.}
\bibitem{W83} \small{M. I. Weinstein,  Nonlinear Schr\"{o}dinger equations and sharp interpolation estimates, \textit{Commun. Math. Phys.}, \textbf{87} (1983), 567--576.}
\bibitem{Y21} \small{K. Yang, Scattering of the focusing energy-critical NLS with inverse-square potential in the radial case, \textit{Comm. Pure Appl. Anal.}, \textbf{20}(1) (2021), 77--99.}
\end{thebibliography}
\end{document}